%% file: geomPerturb_main.tex
\documentclass[11pt]{article}

\include{header-files/header}

\usepackage[%
    minnames=1,maxnames=99,maxcitenames=2,
    style=numeric,%
    sortcites=true,
    doi=false,url=false,
    giveninits=true,
    hyperref,natbib=true,backend=biber]{biblatex}
\renewbibmacro{in:}{%
  \ifentrytype{article}{}{\printtext{\bibstring{in}\intitlepunct}}}
\title{Approximations of Geometrically Ergodic Reversible Markov Chains}
\author{Jeffrey Negrea and Jeffrey S. Rosenthal}
\begin{document}

\maketitle
\abstract{A common tool in the practice of Markov Chain Monte Carlo is to use approximating transition kernels to speed up computation when the desired kernel is slow to evaluate or intractable.
A limited set of quantitative tools exist to assess the relative accuracy and efficiency of such approximations.
We derive a set of tools for such analysis based on the Hilbert space generated by the stationary distribution we intend to sample, $L_2(\pi)$.
Our results apply to approximations of reversible chains which are geometrically ergodic, as is typically the case for applications to Markov Chain Monte Carlo.
The focus of our work is on determining whether the approximating kernel will preserve the geometric ergodicity of the exact chain, and whether the approximating stationary distribution will be close to the original stationary distribution.
For reversible chains, our results extend the results of \citet{johndrow2015approximations} from the uniformly ergodic case to the geometrically ergodic case, under some additional regularity conditions.
We then apply our results to a number of approximate MCMC algorithms.}

\section{Introduction}

\input{./section-files/geomPerturb_intro}

\section{Related Work}
\label{sec-previouswork}

\input{./section-files/geomPerturb_relatedwork}

\section{Perturbation Bounds}
\label{sec-perturbationbounds}

\input{./section-files/geomPerturb_perturbbounds-noproofs}

\section{Applications Markov Chain Monte Carlo}
\label{sec-noisy}
\input{./section-files/geomPerturb_mcmc-noproofs}

\begin{paragraph}{Acknowledgements}
We thank Daniel Rudolf for very helpful comments on the first version of our preprint. We also thank Gareth O. Roberts, Peter Rosenthal, and Don Hadwin for helpful discussions.
\end{paragraph}

\printbibliography

\newpage
\appendix

\section{Proofs}
\label{ax:proofs}

\input{./section-files/geomPerturb_apdx-perturbbounds-proofs}

\section{Proof of  \cref{noisyoperatorthm}}
\label{ax:proofs-noisy}

\input{./section-files/geomPerturb_apdx-mcmc-proofs}

\section{\texorpdfstring{$(L_\infty(\pi),\normpi{\cdot})$}{(L-infinity, L2)}-GE is distinct from $L_2$-GE for non-reversible chains}
\label{ax:weak-not-strong-example}

\input{./section-files/geomPerturb_apdx-example-weakl2rate}

\newpage
\section{Proof of \cref{lem:rt-comment} and \cref{lem:opt-rate-char}}
\label{ax:pf:lem:rt-comment}
\input{./section-files/geomPerturb_apdx-roberts-tweedie-proof}

\end{document}

%% file: header-files/header.tex
\usepackage[margin=1in]{geometry}
\usepackage[utf8]{inputenc}

\usepackage{amsmath,amsthm,wrapfig}
\usepackage{amsmath, amssymb,bm, cases, mathtools, thmtools}

\usepackage{array}
	\newcolumntype{C}[1]{>{\centering}m{#1}}
\usepackage{bbm}
\usepackage{booktabs}
\usepackage{upgreek}
\usepackage{xkvltxp}
\usepackage{fixme}
\usepackage{stmaryrd}
\usepackage{hyperref}
\hypersetup{
    colorlinks=true,
}
\usepackage{multirow}
\usepackage{listings}
\usepackage{xargs}
\usepackage{xstring}
\usepackage{multicol}
\usepackage{graphicx}
\usepackage{float}
\usepackage{color}
\usepackage[usenames,dvipsnames]{xcolor}

\usepackage{enumitem}
\usepackage{algorithm,algpseudocode}
\usepackage{indentfirst}
\usepackage{accents}

\newcommand{\CC}{\mathbbm{C}}

\newcommand{\EE}{\mathbb{E}}

\newcommand{\NN}{\mathbb{N}}

\newcommand{\QQ}{\mathbb{Q}}
\newcommand{\RR}{\mathbb{R}}

\newcommand{\One}{\mathbbm{1}}

\newcommand{\one}{\mathbf{1}}

\newcommand{\zero}{\mathbf{0}}

\newcommand{\Bb}{\mathcal{B}}

\newcommand{\Ee}{\mathcal{E}}

\newcommand{\Ll}{\mathcal{L}}
\newcommand{\Mm}{\mathcal{M}}
\newcommand{\Nn}{\mathcal{N}}

\newcommand{\Xx}{\mathcal{X}}

\newcommand{\Zz}{\mathcal{Z}}

\newcommand{\EEE}[1]{\underset{#1}{\EE}}

\newcommand{\rnderiv}[2]{\frac{\dee #1}{\dee #2}}

\def\[#1\]{\begin{equation}\begin{aligned}#1\end{aligned}\end{equation}}

\def\*[#1\]{\begin{align*}#1\end{align*}}

\newcommand{\st}{\ \text{s.t.}\ }
\newcommand{\andT}{\ \text{and}\ }

\DeclareMathOperator*{\Cov}{Cov}

\DeclareMathOperator*{\sdev}{StdDev}

\DeclareMathOperator{\proj}{proj}

\DeclareMathOperator{\range}{range}

\DeclareMathOperator*{\esssup}{\text{ess}\sup}

\newcommand{\lcr}[3]{\left #1 #2 \right #3}
\newcommand{\lcrx}[4][{-1}]{
	\IfEq{#1}{-1}{\left #2 {{{{#3}}}} \right #4}{
   	\IfEq{#1}{0}{#1 {{{{#2}}}} #3}{
	\IfEq{#1}{1}{\bigl #2 {{{{#3}}}} \bigr #4}{
	\IfEq{#1}{2}{\Bigl #2 {{{{#3}}}} \Bigr #4}{
	\IfEq{#1}{3}{\biggl #2 {{{{#3}}}} \biggr #4}{
	\IfEq{#1}{4}{\Biggl #2 {{{{#3}}}} \Biggr #4}{
    \GenericWarning{"4th argument to lcrx must be -1, 0, 1, 2, 3, or 4"}
    }}}}}}}

\newcommand{\abs}[2][{-1}]{\lcrx[#1] \vert {#2} \vert }
\newcommand{\set}[2][{-1}]{\lcrx[#1] \{ {#2} \}}
\newcommand{\floor}[2][{-1}]{\lcrx \lfloor {#2} \rfloor}
\newcommand{\ceil}[2][{-1}]{\lcrx[#1] \lceil {#2} \rceil}
\newcommand{\norm}[2][{-1}]{\lcrx[#1] \Vert {#2} \Vert}
\newcommand{\Norm}[2][{-1}]{\lcrx[#1] \vert{\kern-0.25ex{\lcrx[#1] \vert{\kern-0.25ex{\lcrx[#1] \vert {#2} \vert}\kern-0.25ex}\vert}\kern-0.25ex}\vert}
\newcommand{\inner}[3][{-1}]{\lcrx[#1] \langle {{#2},\ {#3}} \rangle}

\newcommand{\rbra}[2][{-1}]{\lcrx[#1] ( {#2} ) }
\newcommand{\sbra}[2][{-1}]{\lcrx[#1] [ {#2} ] }

\newcommand{\stk}[2]{\ensuremath{\stackrel{\text{#2}}{#1}}}
\newcommand{\stkm}[2]{\ensuremath{\stackrel{#2}{#1}}}

\newcommand{\normpi}[1]{\norm{#1}_{L_2(\pi)}}
\newcommand{\normpiD}[1]{\norm{#1}_{L'_2(\pi)}}
\newcommand{\normpiDE}[1]{\norm{#1}_{L'_2(\pi_\epsilon)}}
\newcommand{\normpiR}[1]{\norm{#1}_{L_{2,0}(\pi)}}
\newcommand{\normpiO}[1]{\norm{#1}_{L_1(\pi)}}
\newcommand{\normpiInf}[1]{\norm{#1}_{L_\infty(\pi)}}
\newcommand{\normpiE}[1]{\norm{#1}_{L_2(\pi_\epsilon)}}
\newcommand{\normTV}[1]{\norm{#1}_{\text{TV}}}
\newcommand{\dstar}{{{}_{\star\star}}}

\newcommand{\ol}{\overline}

\newcommandx{\dd}[4][1, 2=\mathrm{d}, 4]{\ensuremath{\frac{#2^{#4}#1}{#2#3^{#4}}}}
\newcommandx{\pp}[4][1, 2=\partial, 4]{\ensuremath{\frac{#2^{#4}#1}{#2#3^{#4}}}}

\newcommand{\dee}{\mathrm{d}} %

\DeclareMathOperator*{\newlim}{\mathrm{lim}\vphantom{\mathrm{infsup}}}
\DeclareMathOperator*{\newmin}{\mathrm{min}\vphantom{\mathrm{infsup}}}
\DeclareMathOperator*{\newmax}{\mathrm{max}\vphantom{\mathrm{infsup}}}
\DeclareMathOperator*{\newinf}{\mathrm{inf}\vphantom{\mathrm{infsup}}}
\DeclareMathOperator*{\newsup}{\mathrm{sup}\vphantom{\mathrm{infsup}}}
\renewcommand{\lim}{\newlim}
\renewcommand{\min}{\newmin}
\renewcommand{\max}{\newmax}
\renewcommand{\inf}{\newinf}
\renewcommand{\sup}{\newsup}

\newcommand\Ind[1]{\One_{#1}} %

\newcommand{\union}{\cup}

\newcommand{\Nats}{\NN}

\newcommand{\Rationals}{\QQ}
\newcommand{\Reals}{\RR}

\newcommand{\PosReals}{\Reals_+}
\newcommand{\NatsO}{{\Nats\union\set{0}}}

\usepackage{mathtools}
\def\multiset#1#2{\ensuremath{\left(\kern-.3em\left(\genfrac{}{}{0pt}{}{#1}{#2}\right)\kern-.3em\right)}}

\makeatletter
\newcommand{\subalign}[1]{%
  \vcenter{%
    \Let@ \restore@math@cr \default@tag
    \baselineskip\fontdimen10 \scriptfont\tw@
    \advance\baselineskip\fontdimen12 \scriptfont\tw@
    \lineskip\thr@@\fontdimen8 \scriptfont\thr@@
    \lineskiplimit\lineskip
    \ialign{\hfil$\m@th\scriptstyle##$&$\m@th\scriptstyle{}##$\crcr
      #1\crcr
    }%
  }
}
\makeatother

\definecolor{NavyBlue}{rgb}{0.1,0.1,0.6}

\algrenewcommand\algorithmicrequire{\textbf{Precondition:}}
\algrenewcommand\algorithmicensure{\textbf{Postcondition:}}

\newcommand\restrict[1]{\vert_{#1}}

\newcommand{\tv}{\text{TV}}
\newcommand{\epseps}{\varphi}

\usepackage[capitalize,nameinlink]{cleveref}
\usepackage{crossreftools}
\hypersetup{%
    bookmarksnumbered, bookmarksopen=true, bookmarksopenlevel=1,%
}

\crefname{lemma}{Lemma}{Lemmas}
\crefname{corollary}{Corollary}{Corollaries}
\crefname{theorem}{Theorem}{Theorems}
\makeatletter
\let\reftagform@=\tagform@
\def\tagform@#1{\maketag@@@{\ignorespaces\textcolor{gray}{(#1)}\unskip\@@italiccorr}}
\renewcommand{\eqref}[1]{\textup{\reftagform@{\ref{#1}}}}
\makeatother
\newtheoremstyle{break}
  {\topsep}{\topsep}%
  {\itshape}{}%
  {\bfseries}{}%
  {\newline}{}%
\theoremstyle{break}

\declaretheorem[style=break,numberwithin=section,name=Theorem]{theorem}
\declaretheorem[style=break,sibling=theorem,name=Lemma]{lemma}
\declaretheorem[style=break,sibling=theorem,name=Proposition]{proposition}
\declaretheorem[style=break,sibling=theorem,name=Corollary]{corollary}
\declaretheorem[style=break,sibling=theorem,name=Definition]{definition}
\declaretheorem[style=break,qed=$\triangleleft$,sibling=theorem,name=Example]{example}
\declaretheorem[style=break,qed=$\triangleleft$,sibling=theorem,name=Remark]{remark}

\declaretheorem[name=Assumption, numbered=no]{assumption*}
\numberwithin{equation}{section}
\numberwithin{theorem}{section}

%% file: section-files/geomPerturb_intro.tex
The use of Markov Chain Monte Carlo (MCMC) arises from the need to sample from probabilistic models when simple Monte Carlo is not possible.
The procedure is to simulate a positive recurrent Markov process where the stationary distribution is the measure one intends to sample, so that the dynamics of the process converge to the distribution required.
Temporally correlated samples may then be used to approximate various expectations; see e.g.~\citet{handbook} and the many references therein.
Examples of common applications may be found in hierarchical models, spatio-temporal models, random networks, finance, bioinformatics, etc.

Often, however, the transition dynamics of the Markov Chain required to run this process exactly are too computationally expensive due to prohibitively large datasets, intractable likelihoods, etc.
In such cases it is tempting to instead \emph{approximate} the transition dynamics of the Markov process in question, either deterministically as in the low-rank Gaussian approximation of \citet{johndrow2015approximations}, or stochastically as in the noisy Metropolis--Hastings procedure of \citet{alquier2016noisy}.
It is important then to understand whether these approximations will yield stable and reliable results.
This paper aims to provide quantitative tools for the analysis of these algorithms.
Since the use of approximation for the transition dynamics may be interpreted as a \emph{perturbation} of the transition kernel of the exact MCMC algorithm, we focus on bounds on the convergence of perturbations of Markov chains.

The primary purpose of this paper is to extend existing quantitative bounds on the errors of approximate Markov chains from the uniformly ergodic case in \citep{johndrow2015approximations} to the geometrically
ergodic case (a weaker condition, for which multiple equivalent definitions may be found in \citet{roberts1997geometric}).
Our work will extend the theoretical results of \citep{johndrow2015approximations} in the case that the exact chain is reversible, replacing the total variation metric with $L_2$ distances, and relaxing the uniform contraction condition to $L_2(\pi)$-geometric ergodicity.

\subsection{Geometric Ergodicity}
When analyzing the performance of exact MCMC algorithms, it is natural to decompose the error in approximation of expectations into a component for the transient phase error of the process and one for the Monte-Carlo approximation error.
The former may be interpreted as the bias due to not having started the process
in the stationary distribution.
A Markov chain is \emph{geometrically ergodic} if, from a suitable initial distribution $\nu$, the marginal distribution of the $n^{\text{th}}$ iterate of the chain converges to the stationary distribution, with an error that decays as $C(\nu) \rho^n$ for some $\rho\in(0,1)$ and some constant depending on the initial distribution $C(\nu)$, in some suitable metric on the space of probability measures.
The geometric ergodicity condition essentially dictates that the transient phase error of the $n^{\text{th}}$ sample decays exponentially quickly in $n$.
The chain is \emph{uniformly} (geometrically) ergodic if $C$ can be chosen independently of the initial distribution.
Geometric ergodicity is a desirable property as it ensures that cumulative transient phase error asymptotically does not dominate the Monte-Carlo error, while still being less restrictive than the uniform ergodicity condition, which often fails when the state space is not finite or compact (for example, an AR($1$) process is geometrically ergodic but not uniformly ergodic).

When using approximate MCMC methods, one desires that the approximation
preserves geometric ergodicity, so that convergence to stationarity is still fast
and the transient phase error goes to zero quickly. This is an important issue, especially since \citet{medina2016stability} have shown that intuitive approximations such as Monte-Carlo within Metropolis may lead to transient approximating chains.

\subsection{Outline of the Paper}

The outline of this paper is as follows.
\cref{sec-previouswork} reviews related work.
Then \cref{sec-perturbationbounds} contains our main theoretical results and their proofs.
\cref{combinedthm} therein provides bounds on the distance between stationary distributions, and gives a sufficient condition for the perturbed chain to be geometrically ergodic in $L_2(\pi)$, where $\pi$ is the stationary distribution of the unperturbed chain.
\cref{l1l2geomthm} and \cref{l1l2geomthm-rev} give sufficient conditions for the perturbed chain to be geometrically ergodic according to several other variants of the definition of geometric ergodicity (for different metrics and families of initial distributions), and provide quantitative rates when possible.
The remainder of \cref{sec-perturbationbounds} establishes bounds on autocorrelations, and mean-squared-error for Monte Carlo estimates of expected values computed with the perturbed chain.

Finally, \cref{sec-noisy} considers noisy and/or approximate Metropolis--Hastings algorithms.
It provides sufficient conditions that one can check in order for our results from \cref{sec-perturbationbounds} to be applied.
We use this to study Metropolis--Hastings with deterministic approximations to the target density, as well as the Monte Carlo within Metropolis algorithm, as in \citet{medina2019perturbation}, and provide some examples of how these types of approximations might arise in practice.

%% file: section-files/geomPerturb_relatedwork.tex
This section presents a brief review of related work, discussing convergence of perturbed Markov chains in the uniformly ergodic and geometrically ergodic cases with varying metrics and additional assumptions. The results in the literature have a wide range of assumptions required and a wide range of scopes for their various results. The results for uniformly ergodic chains have a simpler aesthetic, in line with what intuition for finite state space chains might inspire, as they do not require drift and minorization conditions to state. Our results cover the geometrically ergodic and reversible case, and use properties of reversibility to match the simpler aesthetic found in the literature for the uniformly ergodic case.

Close to the present paper, \citet{johndrow2015approximations} derive perturbation bounds to assess the robustness of approximate MCMC algorithms.
The assumptions upon which their results rely are: the original chain is uniformly contractive in the total variation norm (this implies uniform ergodicity); and the perturbation is sufficiently small (in the operator norm induced by the total variation norm). The main results of their paper are:
the perturbed kernel is uniformly contractive in the total variation norm;
the perturbed stationary distribution is close to the original stationary distribution in total variation;
explicit bounds on the total variation distance between finite time approximate sampling distributions and the original stationary distribution;
explicit bounds on total variation difference between the original stationary distribution and the mixture of finite time approximate sampling distributions;
and explicit bounds on the MSE for integral approximation using approximate kernel and the true kernel.
The results derived by \citep{johndrow2015approximations} are applied within the same paper to a wide variety of approximate MCMC problems
including low rank approximation to Gaussian processes and sub-sampling approximations. In other work, \citet{johndrow2017coupling}, use intuitive coupling arguments to establish similar results under the same uniform contractivity assumption.

Further results on perturbations for uniformly ergodic chains may be found in \citet{mitrophanov2005sensitivity}.
This work is motivated in part by numerical rounding errors.
Various applications of these results may be found in \citet{alquier2016noisy}.
The only assumption of \citep{mitrophanov2005sensitivity} is that the original chain is uniformly ergodic.
The paper is unique in that it makes no assumption regarding the proximity of the original and perturbed kernel, though the level of approximation
error does still scale linearly with the total variation distance of the original and perturbed kernels.
The main results are: explicit bounds on the total variation distance between finite time sampling distributions; and explicit bounds on the total variation distance between stationary distributions.

The work of \citet{roberts1998convergence} (see also \citet{breyer2001}) is also motivated by numerical rounding errors.
The perturbed kernel is assumed to be derived from the original kernel by a \textit{round-off function}, which e.g.\ maps the input to nearest multiple of $2^{-31}$.
In such cases, the new state space is at most countable while the old state space may have been uncountable and so the resulting chains have mutually singular marginal distributions at all finite times and mutually singular stationary distributions (if they have stationary distributions at all).
The results of \citep{roberts1998convergence} require the analysis of Lyapunov drift conditions and drift functions (which we will avoid by
working in an appropriate $L_2$ space).
The key assumptions in \citep{roberts1998convergence} are:
the original kernel is geometrically ergodic, and $V$ is a Lyapunov drift function for the original kernel;
the original and perturbed transition kernels are close in the $V$-norm;
the perturbed kernel is defined via a round-off function with round-off error uniformly sufficiently small;
and $\log V$ is uniformly continuous.
The main results of the paper include that:
if the perturbed kernel is sufficiently close in the $V$-norm then geometric ergodicity is preserved;
if the drift function, $V$, can be chosen so that $\log V$ is uniformly continuous and if the round-off errors can be made arbitrarily small then the kernels can be made arbitrarily close in the $V$-norm;
explicit bounds on the total variation distance between the approximate finite-time sampling distribution and the true stationary distribution;
and sufficient conditions for the approximating stationary distribution to be arbitrarily close in total variation to the true stationary distribution.
They also prove results that do not require closeness in the $V$-norm, or even absolute continuity of the perturbed transitions; in such cases they show that a suitable drift condition on the original chain together with a uniformly small round-off error yields perturbed chains which are geometrically ergodic, and that the stationary measure varies continuously under such perturbations in the topology of weak convergence.

\citet{pillai2014ergodicity} provide bounds in terms of the Wasserstein topology (cf.\ \citet{gibbs2004}).
Their main focus is on approximate MCMC algorithms, especially approximation due to sub-sampling from a large dataset (e.g., when computing the posterior density).
Their underlying assumptions are:
the original and perturbed kernels satisfy a series of \textit{drift-like conditions} with shared parameters;
the original kernel has finite eccentricity for all states (where eccentricity of a state is defined as the expected distance between the state and a sample from the stationary distribution);
the \textit{Ricci curvature} of the original kernel has a non-trivial uniform lower bound on a positive measure subset of the state space;
and the transition kernels are close in the Wasserstein metric, uniformly on the mentioned subset.
Their main results under these assumptions are:
explicit bounds on the Wasserstein distance between the approximate sampling distribution and the original stationary distribution;
explicit bounds on the total variation distance of the original and perturbed stationary distributions and bounds on the mixing times of each chain;
explicit bounds on the bias and $L_1$ error of Monte Carlo approximations;
decomposition of the error from approximate MCMC estimation into components from \textit{burn-in}, \textit{asymptotic bias}, and \textit{variance}; and rigorous discussion of the trade-off between the above error components.

\citet{rudolf2015perturbation} also use the Wasserstein topology.
They focus on approximate MCMC algorithms, with applications to auto-regressive processes and stochastic Langevin algorithms for Gibbs random fields.
Their results use the following assumptions:
the original kernel is Wasserstein ergodic;
a Lyapunov drift condition for perturbed kernel is given, with drift function $\tilde{V}$;
$\tilde{V}$ has finite expectation under the initial distribution;
and the perturbation operator is uniformly bounded in a $\tilde{V}$-normalized Wasserstein norm.
Their main results are:
explicit bounds on the Wasserstein distance and weighted total variation distance between the original and perturbed finite time sampling distributions;
and explicit bounds on the Wasserstein distance between stationary distributions.

\citet{ferre2013regular} build upon \citet{keller1999stability} to provide perturbation results for $V$-geometrically ergodic Markov chains using a simultaneous drift condition. They show that any perturbation to the transition kernel which shares its drift condition has a stationary distribution, is also $V$-geometrically ergodic, and that the perturbed stationary distributions is close to the original one. The assumption of a shared drift condition may be difficult to verify or not hold in some cases of interest related to approximate or noisy Markov chain Monte Carlo.
\citet{herve2014approximating} considers finite rank approximations to a transition kernel. That work gives sufficient conditions for approximations to inherit $V$-geometric ergodicity and provides a quantitative relationship between the rates of convergence and bounds the total variation distance between stationary measures. It also provides sufficient conditions for $V$-geometric ergodicity of a family of finite-rank approximations to a transition kernel to guarantee geometric ergodicity of the kernel, and provides a quantitative rates of convergence. In both of these results, as in \citep{ferre2013regular}, the results depend on a simultaneous drift condition for the approximations and the original kernel.

Each of the above papers demonstrate bounds on various measures of error from using approximate finite-time sampling distributions and approximate ergodic distributions to calculate expectations of functions.
On the other hand, the assumptions underlying the results vary dramatically.
The results for uniformly ergodic chains are based on simpler and more intuitive assumptions than those for geometrically ergodic chains.
Our work extends these results to geometrically ergodic chains and perturbations while preserving essentially the same level of simplicity
in the assumptions. In particular we avoid the need to identify a Lyapunov drift condition, and our assumptions are expressed directly in terms of transition kernels, rather than a relationship between drift conditions which they satisfy.

%% file: section-files/geomPerturb_perturbbounds-noproofs.tex
This section extends the main results of \citet{johndrow2015approximations} to the $L_2(\pi)$-geometrically ergodic case for reversible processes, assuming the perturbation $P-P_{\epsilon}$ has bounded $L_{2}(\pi)$ operator norm.

\subsection{Definitions and Notation}

Let $\pi$ be a probability measure on a measurable space $(\Xx, \Sigma)$. We make considerable use of the following norms on signed measures and their corresponding Banach spaces.
\*[
    \normTV{\lambda}
    	& = \sup_{A\in\Sigma} \abs{\lambda(A)}
			&
			\Mm(\Sigma)
	    	& = \set{\text{bounded signed measures on }(\Xx, \Sigma)}
				\\
    \normpi{\lambda}
    	& = \rbra{\int\lcr({\rnderiv\lambda\pi})^2 \dee \pi}^{1/2}
			&
			L_2(\pi)
	    	& = \set{\nu\ll\pi : \normpi{\nu} <\infty }\\
    \normpiR{\cdot}
      	& = \normpi{\cdot}\restrict{L_{2,0}(\pi)}
				&
				L_{2,0}(\pi)
		    	& = \set{\nu\in L_2(\pi) :\nu(\Xx) = 0}\\
    \normpiO{\lambda}
    	& = \int\abs{\rnderiv\lambda\pi} \dee \pi
			&
			L_{1}(\pi)
	    	& = \set{\nu\ll\pi : \normpiO{\nu} <\infty}\\
    \norm{\lambda}_{L_{\infty}(\pi)}
	       & =\esssup_{X\sim\pi}\rnderiv{\lambda}{\pi}(X)
				 &
				 L_{\infty}(\pi)
		 	& = \set{\nu\ll\pi : (\exists b>0)\lcr({\abs{\rnderiv{\nu}{\pi}}<b \quad \pi\text{-a.e.}})}
\]
Note that $L_{2,0}(\pi)$ is a complete subspace of $L_2(\pi)$.
Let
\*[
\Mm_{+,1} = \set{\lambda\in\Mm: [\forall A \in \Sigma\quad  \lambda(A)\geq 0]  \andT [\lambda(\Xx) = 1]}
\]
be the set of probability measures on $(\Xx, \Sigma )$.
Note that for any probability measure, $\pi$, $L_\infty(\pi)\subset  L_2(\pi) \subset L_1(\pi) \subset \Mm(\Sigma)$, though in general they are not complete subspaces of each other when their corresponding norms are not equivalent.
For a norm, $\norm{\cdot}$ on a vector space, we also write $\norm\cdot$ the corresponding operator norm on the space of bounded linear operators from $V$ to itself, $\Bb(V)$.

\begin{definition}[Geometric Ergodicity]\label{def:ge}
Let $P$ be the kernel of a positive recurrent Markov chain with invariant  measure $\pi$. Let $\lambda$ be any measure with $\pi\ll\lambda$. Suppose that $\rho_\tv,\rho_1,\rho_2 \in (0,1)$. Then:
\begin{itemize}
\item[(i)] $P$ is \textbf{$\pi$-a.e.-TV geometrically ergodic with factor $\rho_\tv$} if there exists $C_\tv:\Xx\to\PosReals$ such that for $\pi$-almost every $x\in\Xx$ and for all $n\in\Nats$:
\*[
	\normTV{\delta_x P^n - \pi} \leq C_\tv(x) \rho_\tv^n \ .
\]
The \emph{optimal rate} for $\pi$-a.e.-TV geometric ergodicity is the infimum over factors for which the above definition holds;
\[
	\rho^\star_\tv = \inf\lcr\{{\rho>0 \st }.&\exists C:\Xx\to\PosReals \text{ with }  \pi(\set{x:C(x)<\infty})=1 \andT  \\
	& \lcr.{\forall n\in\Nats, \pi\text{-a.e. }x\in\Xx \quad \normTV{\delta_x P^n - \pi} \leq C(x) \rho^n}\} \ .
\]
\item[(ii)] $P$ is \textbf{$L_2(\lambda)$-geometrically ergodic with factor $\rho_2$} if $P: L_2(\lambda)\to L_2(\lambda)$ and there exists $C_2: L_2(\lambda)\cap\Mm_{+,1}\to\PosReals$ such that for every $\nu\in L_2(\lambda)\cap\Mm_{+,1}$ and for all $n\in\Nats$:
\*[
	\norm{\nu P^n - \pi}_{L_2(\lambda)} \leq C_2(\nu) \rho_2^n \ .
\]
The \emph{optimal rate} for $L_2(\lambda)$-geometric ergodicity is the infimum over factors for which the above definition holds;
\*[
	\rho^\star_2
		& = \inf\lcrx[2]\{{\rho>0 \st \exists C:L_2(\lambda)\cap\Mm_{+,1}\to\PosReals \text{ with } }. \\
		&\qquad \qquad \lcrx[2].{\forall n\in\Nats, \nu\in L_2(\lambda)\cap\Mm_{+,1} \quad \norm{\nu P^n - \pi}_{L_2(\lambda)} \leq C(\nu) \rho^n}\} \ .
\]
\end{itemize}
\end{definition}
\begin{remark}
If $P$ is $\pi$-reversible and aperiodic then $P$ is $L_2(\pi)$-geometrically ergodic if and only if it is $\pi$-a.e. TV geometrically ergodic, as per \citet{roberts1997geometric}. In this case the optimal rate of $L_2(\pi)$-geometric ergodicity, $\rho_2^\star$, is equal to the spectral radius of $P\restrict{L_{2,0}(\pi)}$,
In this case, the spectrum of $P$ is a subset of $[-\rho_2^\star,\rho_2^\star]\union\set{1}$, and $P$ is $L_2(\pi)$-geometrically ergodic with factor $\rho_2^\star$ and $C(\mu) = \normpi{\mu-\pi}$.
For more details see \cref{rr97thm2}, and \citep{roberts1997geometric}.
\end{remark}
We abbreviate \emph{geometric ergodicity} and \emph{geometrically ergodic} as ``GE'' for brevity going forward.

\subsection{Assumptions}
\label{sec-assump}

We assume throughout that $P$ is the transition kernel for a Markov chain on a countably generated state space $\mathcal{X}$ with $\sigma$-algebra $\Sigma$, which is reversible with respect to a stationary probability measure, $\pi$, and is $\pi$-irreducible and aperiodic.
We call the Markov chain induced by $P$ the ``original'' chain.
The $\pi$-reversibility of $P$ makes it natural to work in $L_2(\pi)$ since, in this case, $P$ is a self-adjoint linear operator on a Hilbert space. This allows us access to the rich, elegant, and mature spectral theory of such operators. See for example \citep[Chapter~12]{rudin1991functional} and \citep[Chapter~22]{douc2018markov}.
We further assume that $P$ is $L_2(\pi)$-geometrically ergodic with factor $0<(1-\alpha)<1$.
Equivalent definitions of $L_2(\pi)$-geometrically ergodic are given in \cref{rr97thm2}.
This assumption is weaker than the Doeblin condition used by \citep{johndrow2015approximations}, which implies uniform ergodicity.

Next, we assume that $P_\epsilon$ is a second (``perturbed'') transition kernel, with $\normpi{P-P_\epsilon} \leq \epsilon $ for some fixed $\epsilon>0$, and that $P_\epsilon \restrict{L_2(\pi)} \in \Bb(L_2(\pi))$, i.e. that the perturbed transition kernel maps $L_2(\pi)$ measures to $L_2(\pi)$ measures.
The norm condition quantifies the intuition that the perturbation is ``small''.
We assume that $P_\epsilon$ is $\pi$-irreducible and aperiodic.
We demonstrate (in \cref{combinedthm}) that under these assumptions $P_\epsilon$ has a unique stationary distribution, denoted by $\pi_\epsilon$, with $\pi_\epsilon\in L_2(\pi)$.

Note that when $\mu\in L_1(\pi)$ we have $\normTV{\mu-\pi} =\frac{1}{2}\normpiO{\mu-\pi}$.
On the other hand, $\normTV{\cdot}$ applies to all bounded measures, while $\normpiO{\cdot}$ applies only to the subspace of $L_1(\pi)$ measures.
Note also that if $\pi \sim \pi_\epsilon$ (the two measures are mutually absolutely continuous), then $L_1(\pi)$ and $L_1(\pi_\epsilon)$ are equal as spaces and their norms are always equal, so in this case we need not distinguish between them.

To summarize, we assume that
\begin{assumption*}[Assumptions of \cref{sec-assump}]
\begin{minipage}[h]{0.9\textwidth}
		\begin{multicols}{2}
				\begin{itemize}
					\item $P$ is a Markov kernel that is
					\begin{itemize}
						\item $\pi$-reversible for a prob. meas. $\pi$,
						\item irreducible and aperiodic
						\item $L_2(\pi)$-GE with factor $(1-\alpha)$,
					\end{itemize}
				\end{itemize}
				\begin{itemize}
					\item $P_\epsilon$ is a Markov kernel that is
					\begin{itemize}
						\item irreducible and aperiodic,
						\item $P_\epsilon : L_2(\pi) \to L_2(\pi)$, and
						\item $\normpi{P-P_\epsilon}<\epsilon$.
					\end{itemize}
				\end{itemize}
			\end{multicols}
	\end{minipage}
\end{assumption*}

	The assumption that $P_\epsilon: L_2(\pi)\to L_2(\pi)$ and that $\normpi{P_\epsilon}<\infty$ may seem difficult to verify. However, the following proposition shows us that it is satisfied for $P_\epsilon$ constructed based on the Metropolis--Hastings algorithm with suitable \emph{jump kernels}. As long as the jump kernel, $J$, has $\normpi{J}<\infty$ then it will be satisfied. Therefore, this assumption is not excessively restrictive for MCMC applications.
  The jump kernel, $J$, describes the conditional distribution of a new point in the chain proposed from $x$ given that the proposal is accepted, and is related to the proposal kernel, $Q$, by $\alpha(x) J(x,A) = \int_A a(x,y) Q(x, dy) $ where $a(x,y)$ is the Metropolis--Hastings acceptance ratio and $\alpha(x) = \int_\Xx a(x,y) Q(x, dy)$ is the implied local jump-intensity.

	\begin{proposition}
		\label{prop:mh-perturb-endomorphism}
		If $P_\epsilon(x,\cdot) = (1-\alpha(x))\delta_x +\alpha(x)J(x,\cdot)$ with $\alpha:\Xx\to[0,1]$ measurable, and $J : L_2(\pi)\to L_2(\pi)$ and $\normpi{J}<\infty$, then
    \[\normpi{P_\epsilon} \leq 1 +\normpi{J}.\]

	\end{proposition}

	\begin{proof}
		[Proof of \cref{prop:mh-perturb-endomorphism}]
    Consider the operator $A$ on $L_2(\pi)$ given by $[\nu A](C) = \int_C \alpha(x)\nu(dx)$ for all measurable sets $C$.
    Its adjoint, $A'$, is given by $[A' f] (x) = \alpha(x) f(x)$ for all $x \in \Xx$ and $f\in L_2'(\pi)$.
    Since $\alpha : \Xx \to [0,1]$, then
    $A':L_2'(\pi)\to L_2'(\pi)$ with $\normpiD{A'} \leq 1$.
    Thus $A : L_2(\pi)\to L_2(\pi)$ with $\normpi{A} \leq 1 $. The same also holds for $I-A$.
    Now, $P_\epsilon = A + (I-A) J$, so $\normpi{P_\epsilon} \leq 1 +\normpi{J}.$
	\end{proof}
	Verifying that $\normpi{P-P_\epsilon}$ is finite, and sufficiently small will be the main analytic burden faced when trying to apply our results to more general settings. The development of further tools to determine whether $\normpi{P-P_\epsilon}$ is finite and to bound it quantitatively would be an interesting line of future research.

\subsection{Convergence Rates and Closeness of Stationary Distributions}

\begin{theorem}[Geometric ergodicity of the perturbed chain and closeness of the stationary distributions in \emph{original norm}, $L_2(\pi)$]
\label{combinedthm}%
Under the assumptions of \cref{sec-assump},
if in addition $\epsilon<\alpha$, then $\pi_\epsilon\in L_2(\pi)$,
\*[0\leq\normpi{\pi-\pi_\epsilon} \leq
\frac{\epsilon}{\sqrt{\alpha^2-\epsilon^2}} \ , \]
$P_\epsilon$ is $L_2(\pi)$-geometrically
ergodic with factor $1-(\alpha-\epsilon)$, and for any initial probability measure $\mu\in L_2(\pi)$
\*[
	\normpi{ \mu P_\epsilon^n -\pi}
		& \leq (1-(\alpha-\epsilon))^n \normpi{\mu-\pi_\epsilon}  + \frac{\epsilon}{\sqrt{\alpha^2-\epsilon^2}} \ ,
\]
\end{theorem}
The proof of this result is the content of \cref{ax:proof-combinedthm}.
We follow the derivation in \citep{johndrow2015approximations} with minimal structural modification, though the technicalities must be handled differently and additional theoretical machinery is required.
We use the fact that the existence of a spectral gap for the restriction of $P$ to $L_{2,0}(\pi)$ yields an inequality of the same form as uniform contractivity condition, but in the $L_{2}(\pi)$-norm as opposed to the total variation norm (cf.\ Theorem~2.1 of \citet{roberts1997geometric}).

\begin{remark}
	Bounds on the differences between measures in $L_2(\pi)$-norm can be converted into bounds on the total variation distance since, by Cauchy-Schwarz, for any measure $\lambda$ and any signed measure $\nu \in L_2(\lambda)$ we have $\normTV{\nu} = \frac{1}{2}\norm{\nu}_{L_1(\lambda)} \leq \frac{1}{2}\norm{\nu}_{L_2(\lambda)}$. Thus, for example, under the assumptions of \cref{combinedthm},
	\*[
		\normTV{ \mu P_\epsilon^n -\pi}\leq \frac{1}{2}\sbra{(1-(\alpha-\epsilon))^n \normpi{\mu-\pi_\epsilon}  + \frac{\epsilon}{\sqrt{\alpha^2-\epsilon^2}}}.
	\]
	Similarly, under the assumptions of \cref{combinedthm}, we find that $P_\epsilon$ is $(L_2(\pi),\normTV{\cdot})$-GE with factor $1-(\alpha-\epsilon)$ (see \cref{def:mixed-GE} below).
\end{remark}

In some situations, such as the computation of mean-squared errors in \cref{mse_perturb}, it may be inconvenient or impossible to use to use the $L_2(\pi)$ norm when studying some aspects of $P_\epsilon$. The next theorem will allow us to ``switch'' to other norms which may be more natural for a given task. First, however, we need to introduce one more notion of geometric ergodicity.

\begin{definition}[$(V,\Norm{\cdot})$-Geometric Ergodicity]\label{def:mixed-GE}
Let $P$ be the kernel of a positive recurrent Markov chain with invariant measure $\pi$. Let $V$ be a vector space of signed measures on $(\Xx,\Sigma)$ containing $\pi$, and let $\Norm{\cdot}$ be a norm on $V$ (for which $V$ may not be complete).

$P$ is \textbf{$(V,\Norm{\cdot})$-geometrically ergodic with factor $\rho$} if there exists $C: V\cap \Mm_{+,1}\to\PosReals$ such that for every $\nu\in V\cap \Mm_{+,1}$ and for all $n\in\Nats$:
  \*[
  	\Norm{\nu P^n - \pi}\leq C(\nu) \rho^n \ .
  \]
	The \emph{optimal rate} for $(V,\Norm{\cdot})$-geometric ergodicity is the infimum over factors for which the above definition holds;
	\*[
		\rho^\star = \inf\lcr\{{\rho>0: \exists C:V\cap \Mm_{+,1}\to\PosReals \st \forall n\in\Nats, \nu\in V\cap \Mm_{+,1} \quad \Norm{\nu P^n - \pi} \leq C(\nu) \rho^n}\} \ .
	\]
\end{definition}
We will be interested in this definition for the cases that $V = L_\infty(\pi)$ and $\Norm{\cdot}$ is either $\normpi{\cdot}$ or $\normpiO{\cdot}$.

\begin{remark}[Relationships between $(L_\infty(\lambda),\norm{\cdot}_{L_p(\lambda)})$-GE, a.e.-TV-GE,  and $L_2(\lambda)$-GE]
Clearly if $P$ is $L_2(\lambda)$-GE with factor $\rho_2$ then it is also $(L_\infty(\lambda),\norm{\cdot}_{L_2(\lambda)})$-GE with factor $\rho_2$.
Conversely \citet{roberts2001geometric} show that if $P$ is $(L_\infty(\pi),\norm{\cdot}_{L_2(\pi)})$-GE with factor $\rho_2$ then it is also a.e.-TV-GE with some factor $\rho_\tv\in(0,1)$.
However the factor for a.e.-TV-GE may in fact be worse than the factor of $(L_\infty(\pi),\norm{\cdot}_{L_2(\pi)})$-GE or $(L_\infty(\pi),\norm{\cdot}_{L_1(\pi)})$-GE.
\citet{baxendale2005renewal} gives a detailed exposition on the barriers to the comparison of factors for geometric ergodicity given by different equivalent definitions.

In \cref{ax:weak-not-strong-example} we give an example where the optimal rates for $L_2(\pi)$-GE and $(L_\infty(\pi),\norm{\cdot}_{L_2(\pi)})$-GE are distinct when $P$ is not reversible.
If $P$ is $\pi$-reversible then the factors for $L_2(\pi)$-GE, $(L_\infty(\pi),\norm{\cdot}_{L_2(\pi)})$-GE, and $(L_\infty(\pi),\norm{\cdot}_{L_1(\pi)})$-GE must be the same. This result combines a comment and Theorem 3 of \citep{roberts2001geometric}, both stated but not proved.
The formal statement of that result and its proof may be found in \cref{ax:pf:lem:rt-comment}.

Finally, note that by definition $L_2(\pi)$-GE is equivalent to $(L_2(\pi),\normpi{\cdot})$ with the same coefficient functions and factors, and that a.e.-TV-GE is equivalent to $(D,\normTV{\cdot})$-GE where we can take $D = \text{span}\rbra{\set{\pi}\union\set{\delta_x: x\in \Xx\setminus{N}, r\in\Reals}}$ for some $\pi$-null set $N$. The null set, $N$, can be taken to be the same for all factors $\rho$ by taking the union over the null sets for factors $\rho\in\Rationals$ (since a countable union of null sets is still null).
\end{remark}

\begin{lemma}[Characterization of optimal rates for $(V,\Norm{\cdot})$-GE chains]
	\label{lem:opt-rate-char}
	If $P$ is $(V,\Norm{\cdot})$-GE with stationary measure $\pi$ then the optimal rate for $(V,\Norm{\cdot})$-GE is equal to
	\[
		\sup_{\mu\in V\cap \Mm_{+,1}} \limsup_{n\to\infty} \Norm{\mu P^n -\pi}^{1/n} \ .
	\]
\end{lemma}
The proof of this result is found in \cref{ax:pf:lem:rt-comment}.

\begin{remark}
	The quantity $\limsup_{n\to\infty} \Norm{\mu P^n -\pi}^{1/n}$ is the \emph{local spectral radius} of $P-\Pi$ at $\mu$ with respect to $\Norm{\cdot}$, where $\Pi$ is the rank-1 kernel defined by $\Pi(x,A) = \pi(A)$ for all $x\in\Xx$ and $A\in \Sigma$.
\end{remark}

\begin{lemma}[$L_2(\pi)$-GE, $(L_\infty(\pi), \norm{\cdot}_{L_2(\pi)})$-GE, and $(L_\infty(\pi), \norm{\cdot}_{L_1(\pi)})$-GE are equivalent for $\pi$-reversible chains, with equal optimal rates.]\label{lem:rt-comment}
    Let $\rho\in[0,1)$. The following are equivalent for a $\pi$-reversible Markov Chain $P$:
    \begin{enumerate}
      \item[(i)] $P$ is $(L_\infty(\pi), \norm{\cdot}_{L_1(\pi)})$-geometrically ergodic with optimal rate $\rho$,
      \item[(ii)] $P$ is $(L_\infty(\pi), \norm{\cdot}_{L_2(\pi)})$-geometrically ergodic with optimal rate $\rho$,
      \item[(iii)] $P$ is $L_2(\pi)$-geometrically ergodic with optimal rate $\rho$,
      \item[(iv)] The spectral radius of $P\restrict{L_{2,0}(\pi)}$ is equal to $\rho$.
    \end{enumerate}
\end{lemma}
\begin{remark}
  Since either of \textit{(iii)} or \textit{(iv)} are equivalent to all the conditions listed in \citet[Theorem 2.1]{roberts1997geometric}, indeed all of the items listed above are equivalent to all the items listed in their result. We only included \textit{(iii)} and \textit{(iv)} here for brevity, and since they are the ones most relevant to the present paper. Moreover, all of these conditions are implied by any of the equivalent conditions for $\pi$-a.e.-TV-GE in \citet[Proposition 2.1]{roberts1997geometric} (though with possibly different optimal rates for each condition therein).
\end{remark}
The proof of this result is found in \cref{ax:pf:lem:rt-comment}.

\cref{combinedthm} controls the convergence of the perturbed chain $P_\epsilon$ in terms of the ``original'' norm (from $L_2(\pi)$).
We also demonstrate that $P_\epsilon$ is geometrically ergodic in the $L_2(\pi_\epsilon)$ norm, as this would also allow us to use the equivalences in \citep{roberts1997geometric}.
The following two results allow us to transfer the geometric ergodicity of $P_\epsilon$ in $L_2(\pi)$ to other notions of geometric ergodicity.
\cref{l1l2geomthm-rev} handles the case that the perturbed kernel is reversible, while \cref{l1l2geomthm} handles both that the perturbed kernel is reversible or non-reversible.

\begin{theorem}[Geometric ergodicity of the perturbed chain in the \emph{other norms}; $L_1(\pi_\epsilon)$, $L_2(\pi_\epsilon)$, total variation]
\label{l1l2geomthm}%
Under the assumptions of \cref{sec-assump}, if $\epsilon<\alpha$, then:
\begin{itemize}
\item[(i)] $P_\epsilon$ is a.e.-TV-geometrically ergodic with some factor $\rho_\tv \in (0,1)$, and
\item[(ii)] $P_\epsilon$ is $(L_\infty(\pi_\epsilon),\norm{\cdot}_{L_1(\pi_\epsilon)})$-GE with factor $\rho_1=(1-(\alpha-\epsilon))$ and $C_1(\mu) = \normpi{\mu-\pi_\epsilon}$, and
\item[(iii)] If $\pi\in L_\infty(\pi_\epsilon)$ then $P_\epsilon$ is $L_2(\pi_\epsilon)$-GE with factor $\rho_2=(1-(\alpha-\epsilon))$ and \*[C_2(\mu) = \norm{\pi}_{L_\infty(\pi_\epsilon)}^{1/2}\normpi{\mu-\pi} \ .\]
\end{itemize}
\end{theorem}
The proof of this result is found in \cref{ax:proofs-l1l2-margerr}.

\begin{example}
	For example, consider perturbations of a Gaussian $\text{AR}(1)$ process. Let $Z_i\stk\sim{iid} \Nn(0,\sigma^2)$ and let $W_i\stk\sim{iid} \mu$.
	Take
	\[
	    X_{t+1} | X_t
	        & = (1-\alpha) X_t + Z_{t+1}\\
	    X_{t+1}^\epsilon|X_t^\epsilon
	        & = (1-\alpha) X_t^\epsilon + W_{t+1} .
	\]
	Then the original chain, $\set{X_t}_{t\in\Nats}$ is not uniformly ergodic, but it is geometrically ergodic. Hence, the results of \citep{alquier2016noisy,johndrow2015approximations} do not apply. The stationary measure of the exact chain is
	$\pi\equiv \Nn(0,\frac{\sigma^2}{\alpha(2-\alpha)})$, it is reversible, and the rate of geometric ergodicity is $(1-\alpha)$.
	Note that the perturbed chain, which we will call a $\mu$-$\text{AR}(1)$ process, may not be reversible and whether it is geometrically ergodic generally depends on the distribution $\mu$.

	Now, letting $\phi_{\sigma^2}$ be the $\Nn(0,\sigma^2)$ density, for any $\mu$ with $\rnderiv{\mu}{\phi_{\sigma^2}}\in [1-\epsilon,1+\epsilon]$, ,
	\[
	    \normpi{P - P_\epsilon}^2
	        & = \int_{-\infty}^\infty \int_{-\infty}^\infty  \rbra{ \frac{\mu(y-(1-\alpha)x)}{\pi(y)} - \frac{\phi_{\sigma^2}(y-(1-\alpha)x)}{\pi(y)}}^2 \pi(y)dy\ \pi(x)dx \\
	        & \leq \int_{-\infty}^\infty\int_{-\infty}^\infty  \epsilon^2 \rbra{\frac{\phi_{\sigma^2}(y-(1-\alpha)x)}{\pi(y)} dy}^2 \pi(y)dy\ \pi(x)dx \\
	        & = \epsilon^2 \normpi{P}\\
	        & = \epsilon^2
	\]
	Therefore, when $\epsilon<\alpha$ we can extend the geometric ergodicity of the Gaussian AR process to the $\mu-\text{AR}(1)$ process using \cref{l1l2geomthm}.
	We can also bound the discrepancy of the stationary measure of the perturbed chain from that $\Nn(0,\frac{\sigma^2}{\alpha(2-\alpha)})$ using \cref{combinedthm}.
	The subsequent results, \cref{cor:marg-error-bd,perturberrorthm} of this section may also be applied to this example to bound the discrepancy between the marginal distributions of the $\mu$-$\text{AR}(1)$ from a $\Nn(0,\frac{\sigma^2}{\alpha(2-\alpha)})$ at any time, as well as the approximation error of the time-averaged law of the $\mu$-$\text{AR}(1)$ from $\Nn(0,\frac{\sigma^2}{\alpha(2-\alpha)})$ .
\end{example}

\begin{theorem}[$L_2(\pi_\epsilon)$-Geometric ergodicity of the perturbed chain, reversible case]
\label{l1l2geomthm-rev}%
Under the assumptions of \cref{sec-assump}, if $\epsilon<\alpha$, and $P_\epsilon$ is $\pi_\epsilon$-reversible, then $P_\epsilon$ is $L_2(\pi_\epsilon)$-GE with factor $\rho_2 = (1 - \alpha + \epsilon)$ and coefficient function $C(\nu) = \normpiE{\nu}$.
\end{theorem}
The proof of this result is found in \cref{ax:proofs-l1l2-margerr}.

\begin{corollary}[Closeness of stationary distributions in $L_2(\pi_\epsilon)$]
\label{cor:marg-error-bd}
If $\epsilon < \alpha$, and $\norm{P-P_\epsilon}_{L_2(\pi_\epsilon)} \leq \epseps$  then
\begin{itemize}
\item[(i)] if $P_\epsilon$ is $\pi_\epsilon$ reversible, and if $\epseps<\alpha-\epsilon$ then
\*[
	\normpiE{\pi-\pi_\epsilon}
	\leq \frac{\epseps}{\sqrt{(\alpha-\epsilon)^2-\epseps^2}} \ ,
\]
and for any $\mu\in L_2(\pi_\epsilon)$
\*[
	\normpiE{ \mu P_\epsilon^n -\pi}
		& \leq (1-(\alpha-\epsilon))^n \normpiE{\mu-\pi_\epsilon}  + \frac{\epseps}{\sqrt{(\alpha-\epsilon)^2-\epseps^2}} \ ,
\]
\item[(ii)] if $\pi\in L_\infty(\pi_\epsilon)$ and $\epseps<1$, then
\*[
	\normpiE{\pi-\pi_\epsilon}
		\leq \frac{\epseps + \norm{\pi}_{L_\infty(\pi_\epsilon)}^{1/2}\frac{\epsilon}{\sqrt{\alpha^2-\epsilon^2}}(1-(\alpha-\epsilon))}{1-\epseps} \ ,
\]
and for any $\mu\in L_\infty(\pi_\epsilon)$
\*[
	\normpiE{ \mu P_\epsilon^n -\pi}
		& \leq (1-(\alpha-\epsilon))^n \norm{\pi}_{L_\infty(\pi_\epsilon)}^{1/2} \normpiE{\mu-\pi_\epsilon}\\
		&\qquad  + 	\frac{\epseps + \norm{\pi}_{L_\infty(\pi_\epsilon)}^{1/2}\frac{\epsilon}{\sqrt{\alpha^2-\epsilon^2}}(1-(\alpha-\epsilon))}{1-\epseps} \ ,
\]
\end{itemize}
\end{corollary}
The proof of this result is found in \cref{ax:proofs-l1l2-margerr}.
We turn our attention to bounds on the error of estimation
measures of the form $\frac{1}{t} \sum_{k=0}^{t-1} \mu P^k$, and
estimates of the form $\frac{1}{t} \sum_{k=0}^{t-1} f(X_k)$.
Firstly, when computing Monte Carlo estimates, the bias is controlled by a time-averaged marginal distribution of the form $\frac{1}{t}\sum_{k=0}^{t-1} \mu P_\epsilon^k $. This leads us to the following result.

\begin{theorem}[Convergence of Time-Averaged Marginal Distributions]
\label{perturberrorthm}%
Under the assumptions of \cref{sec-assump},
suppose $\epsilon<\alpha$ and $\pi_\epsilon\in L_2(\pi)$. Then for any
probability distribution $\mu\in L_2(\pi)$,
\*[
\normpi{\pi -\frac{1}{t}\sum_{k=0}^{t-1} \mu P_\epsilon^k}
    & \leq \frac{1-(1-(\alpha-\epsilon))^t}{t(\alpha-\epsilon)}\normpi{
    \pi_\epsilon-\mu}
    +\frac{\epsilon}{\sqrt{\alpha^2-\epsilon^2}}
\]

If additionally, $\norm{P-P_\epsilon}_{L_2(\pi_\epsilon)} \leq \epseps$  then
\begin{itemize}
\item[(i)] if  $P_\epsilon$ is $\pi_\epsilon$-reversible, and $\epseps<\alpha-\epsilon$ then
\*[
\normpiE{\pi -\frac{1}{t}\sum_{k=0}^{t-1} \mu P_\epsilon^k}
    &\leq \frac{1-(1-(\alpha-\epsilon))^t}{t(\alpha-\epsilon)}\normpiE{
    \pi_\epsilon-\mu}
    +\frac{\epseps}{\sqrt{(\alpha-\epsilon)^2-\epseps^2}}
\]

\item[(ii)] if $\pi\in L_\infty(\pi_\epsilon)$ and $\epseps<1$, and if $\mu\in L_\infty(\pi_\epsilon)$ then
\*[
\normpiE{\pi -\frac{1}{t}\sum_{k=0}^{t-1} \mu P_\epsilon^k}
    &\leq \frac{1-(1-(\alpha-\epsilon))^t}{t(\alpha-\epsilon)}\norm{\pi}_{L_\infty(\pi_\epsilon)}^{1/2}\normpiE{
    \pi_\epsilon-\mu}\\
    &\qquad +\frac{\epseps + \norm{\pi}_{L_\infty(\pi_\epsilon)}^{1/2}\frac{\epsilon}{\sqrt{\alpha^2-\epsilon^2}}(1-(\alpha-\epsilon))}{1-\epseps}
\]
\end{itemize}
\end{theorem}
The proof of this result is found in \cref{ax:time-avg-proofs}.
Relative to the uniform closeness of kernels (in total variation) required \citep{johndrow2015approximations}, our assumption that the approximating kernel is close in the operator norm induced by $L_2(\pi)$ is non-comparable. This is because our bound is in terms of the $L_2$ distance which always upper-bounds the total variation distance (up to a constant factor of $1/2$), but our assumption also does not require spatial uniformity which \citep{johndrow2015approximations}'s does.
Thus, this paper's assumptions are not weaker nor stronger than those in \citep{johndrow2015approximations}. Comparing the above results to the corresponding $L_1$ result of \citep{johndrow2015approximations}, we see that the transient phase bias part of our $L_2$ bounds differ from their $L_1$ transient phase bias bound only by a factor which is constant in time, but varies with the initial distribution (as is to be expected when moving from uniform ergodicity to geometric ergodicity).

\subsection{Mean Squared Error Bounds for Monte Carlo Estimates}
Suppose that $\rbra{X^\epsilon_k}_{k\in\NatsO}$ is a realization of the Markov chain with transition kernel $P_\epsilon$ and initial distribution $\mu$.
\newcommand{\MSE}[4]{\text{MSE}^{#1}_{#2}\rbra{#3,#4}}
The mean squared error of a Monte Carlo estimate of $\pi f$ made using $\rbra{X^\epsilon_k}_{k\leq t}$ is given by
\[
	\MSE{\epsilon}{t}{\mu}{f}
		& = \mathbb{E}\left[\left(\pi(f) - \frac{1}{t}\sum_{k=0}^{t-1} f(X^\epsilon_k)\right)^2\right]
\]

\begin{theorem}[Mean Squared Error of Monte Carlo Estimates from the Perturbed Chain]
\label{mse_perturb}
Under the assumptions of \cref{sec-assump}, if $\epsilon<\alpha$, $X^\epsilon_0\sim\mu$, $P_\epsilon$ is $\pi_\epsilon$-reversible, and $\rho_2 = (1-(\alpha-\epsilon))$
then for $f\in L'_4(\pi_\epsilon)$
\begin{itemize}
	\item[(i)] if $f\in L_2'(\pi)$ as well, then
	\*[
	\MSE{\epsilon}{t}{\mu}{f}
			&\leq \frac{2\normpiDE{ f - \pi_\epsilon f}^2}{(1-\rho_2) t}
							+ \frac{2^{7/2} \normpiE{ \mu -\pi_\epsilon} \norm{f -\pi_\epsilon f }_{L_4'(\pi_\epsilon)}^2}{(1-\rho_2)^2 t^2} \\
					& \qquad + \normpiD{f-\pi_\epsilon f}^2\lcr({\frac{\epsilon^2}{\alpha^2 - \epsilon^2}  + 2\frac{\epsilon}{\sqrt{\alpha^2 - \epsilon^2}}\frac{1}{t(\alpha-\epsilon)}\normpi{\pi_\epsilon-\mu}})
	\]
	and
	\*[
	\MSE{\epsilon}{t}{\mu}{f}
			&\leq \frac{4\normpiDE{ f - \pi_\epsilon f}^2}{(1-\rho_2) t}
						+ \frac{2^{9/2} \normpiE{ \mu -\pi_\epsilon} \norm{f -\pi_\epsilon f }_{L_4'(\pi_\epsilon)}^2}{(1-\rho_2)^2 t^2} \\
					&\qquad+ 2\normpiD{f-\pi_\epsilon f}^2\frac{\epsilon^2}{\alpha^2 - \epsilon^2} \ ,
	\]
	and
\item[(ii)] if $\normpiE{P-P_\epsilon}\leq \epseps < (1-\rho_2)$, then
\*[
\MSE{\epsilon}{t}{\mu}{f}
		&\leq \frac{2^{7/2} \normpiE{ \mu -\pi_\epsilon} \norm{f -\pi_\epsilon f }_{L_4'(\pi_\epsilon)}^2}{(1-\rho_2)^2 t^2} \\
				& \qquad + \lcr.{\normpiDE{f-\pi_\epsilon f}^2}.\lcr({\frac{\epseps^2}{(1-\rho_2)^2 - \epseps^2}+ 2\frac{1+\frac{\epseps}{\sqrt{(1-\rho_2)^2 - \epseps^2}}}  {t(1-\rho_2)}\normpiE{
								\pi_\epsilon-\mu}}),
\]
and
\*[
\MSE{\epsilon}{t}{\mu}{f}
		&\leq \frac{2^{9/2} \normpiE{ \mu -\pi_\epsilon} \norm{f -\pi_\epsilon f }_{L_4'(\pi_\epsilon)}^2}{(1-\rho_2)^2 t^2} \\
				& \qquad + \lcr.{\normpiDE{f-\pi_\epsilon f}^2}.\lcr({\frac{2\epseps^2}{(1-\rho_2)^2 - \epseps^2}+ \frac{4}  {t(1-\rho_2)}\normpiE{
								\pi_\epsilon-\mu}})
\]
\end{itemize}

\end{theorem}
The proof of this result is found in \cref{ax:mse-proofs}.
Perturbation bounds based upon drift and minorization conditions could provide similar MSE bounds for functions in $L_2(\pi_\epsilon)$ with $\sup_{x\in\Xx} \frac{\abs{f}}{\sqrt{V}} <\infty$ (where $V$ is the function appearing in the drift condition), as in the work of \citet{johndrow2017error}.
While that may be a larger class of functions than $L_4'(\pi_\epsilon)$ (depending on what $V$ happens to be), the class $L_4'(\pi_\epsilon)$ is quite rich making this bound still useful.
Moreover, the class of functions to which our MSE bounds apply, and the value of the bound itself, depend only on intrinsic features of the Markov chains under consideration. In contrast bounds based on drift and minorization conditions include extrinsic features---introduced by the user for analytic purposes (such as the drift function, $V$)---of which many choices might exist; each leading to different function classes and different bounds.

%% file: section-files/geomPerturb_mcmc-noproofs.tex
In this section we apply our theoretical results to some specific variants of Markov Chain Monte Carlo (MCMC) algorithms to obtain guarantees for noisy and/or approximate variants of MCMC algorithms.
MCMC is used to generate (correlated) samples approximately from a target distribution for which the (unnormalized) density can be evaluated. The key insight is to construct a (typically reversible) Markov chain for which the stationary distribution is the target distribution. This is possible since the reversibility condition is readily verified locally (without integration).

The most commonly used family of MCMC methods is the Metropolis--Hastings algorithm (MH). The chain is initialized from some distribution $X_0\sim \mu_0$. At each step a \emph{proposal} is drawn from some transition kernel, $Y_t \sim Q(X_{t-1},\cdot)$. Suppose that the kernel $Q(x,\cdot)$ has density $q(\cdot\vert x)$. The proposal is \emph{accepted} with probability $a(Y_t|X_{t-1}) =\min\rbra{1,\frac{\pi(Y_t)
q(X_{t-1}\vert Y_t)}{\pi(X_{t-1}) q(Y_t \vert X_{t-1})}}$. If the proposal is accepted then $X_t = Y_t$, and if it is \emph{rejected} (not accepted) then $X_t = X_{t-1}$. The combination of proposal and accept/reject steps yields a $\pi$-reversible Markov kernel, and reversibility guarantees that the stationary distribution is the target distribution.  The user has freedom in selecting the proposal kernel, $Q$, and some choice lead to better performance than others. The accept/reject step requires evaluating the target density, $\pi$, twice on each step.

A large body of research exists guaranteeing that specific MCMC algorithms will be geometrically ergodic (see for example \citep{livingstone2019geometric,hobert1998geometric,roberts2004general}, and many more.). These typically verify geometric ergodicity for a collection of target distributions, $\pi$, and for a small family of proposal kernels, $Q$.

If the target likelihood involves some integral which is computed numerically or by simple Monte Carlo then the numerical and/or stochastic approximation introduces a \emph{perturbation} to the idealized MCMC scheme. This occurs even in standard and widely used statistical models such as generalized linear mixed effect models (GLMMs), since the random effects are nuisance variables which need to be integrated away, either using Laplace or Gaussian quadrature schemes, or by simple Monte Carlo, in order to evaluate the likelihood. Since the Metropolis--Hastings algorithm requires evaluation of the density, these each introduce a perturbation in the acceptance ratio, and hence in the actual transition kernel of the MH scheme.
We now consider the
extent to which our results from Section 3 can be applied to prove
geometric ergodicity for certain approximate MCMC algorithms.

\subsection{Noisy and Approximate MCMC}
The noisy (or approximate) Metropolis--Hastings algorithm (nMH), as found in
\citet{alquier2016noisy} (see also \citet{medina2016stability}) was briefly described above.
The algorithm is defined exactly the same way as the Metropolis--Hastings algorithm, except that the \emph{acceptance ratio}, $a(Y_t|X_{t-1})$, is replaced by a (possibly stochastic) approximation $\widehat a(Y_t|X_{t-1}, Z_t)$.
Here $Z_t$ denotes some random element providing an additional source of randomness, so that $a(Y_t|X_{t-1}, Z_t)$ is not $\sigma(Y_t, X_{t-1})$-measurable when the approximation $\widehat a(Y_t|X_{t-1}, Z_t)$ is stochastic. In the case of a deterministic approximation, $Z_t$ can be ignored or treated as a constant. The approximation can typically be though of as replacing the target density in the acceptance ratio with some approximation.
This includes most approximate MCMC algorithms which preserve the state space and the Markov property, such as replacing $\pi$ with a deterministic approximation or and independent stochastic approximation at each step (as in Monte Carlo within Metropolis).
It does not include algorithms which retain the Markov property only an augmented state space, such as the Pseudo-Marginal approach of \citet{andrieu2009pseudo}.

For our analysis of these algorithms, $P$ will represent the transition
kernel for the MH algorithm while $\widehat P$ will represent the
kernel for the corresponding nMH chain.
The key step in applying our results from \cref{sec-perturbationbounds} will be to show the $L_2(\pi)$ closeness of the nMH transition kernel to the MH transition kernel.
Again, $\normpi{\cdot}$ is the norm on $L_2(\pi)$ and the corresponding operator norm.
We will assume that $\pi$ and $\{Q(x,\cdot)\}_{x\in \Xx}$ are all absolutely
continuous with respect to the Lebesgue measure and have densities
$\pi$ and $\set{q(\cdot\vert x)}_{x\in\Xx}$ respectively.
All arguments used would still apply if there were an arbitrary dominating measure in place of the Lebesgue measure.
Let $F_{y\vert x}$ be the regular conditional distribution for $Z$ given $X=x$ and $Y=y$, and let $f_{y\vert x}$ be its Lebesgue density.
Define the following \emph{perturbation function} for the nMH algorithm as
\*[
r(y|x)
    &= \EEE{Z\sim F_{y\vert x}} \left(
    {{a}}(y|x)-{\widehat{a}}(y|x,Z)\right)
    = \int  \left( {{a}}(y|x)-{\widehat{a}}(y|x,z)\right)
    f_{y\vert x}(z) dz
\]

\begin{theorem}[Geometric ergodicity and closeness of stationary distributions noisy or approximate Metropolis--Hastings]
\label{noisyoperatorthm}
Let $P$ be the transition kernel for a Metropolis--Hastings algorithm with proposal distribution $Q$, target distribution $\pi$, and acceptance ratio $a(\cdot\vert \cdot)$.
Let $\widehat P$ be the transition kernel for a corresponding noisy Metropolis--Hastings algorithm with approximate/noisy acceptance ratio $\widehat a(\cdot\vert \cdot, \cdot)$. Let $r(\cdot\vert \cdot)$ be the corresponding perturbation function.

If $\normpi{Q}<\infty$ and $\sup_{x,y} \abs{r(y|x)}\leq R$ then
\[\normpi{ \widehat P - P}\leq R (1+\normpi{Q})\ .\]
Furthermore, if $P$ is reversible and $L_2(\pi)$-geometrically ergodic with geometric contraction factor $(1-\alpha)$, and $\epsilon = R (1+\normpi{Q}) <\alpha$, then $\widehat P$ has a stationary distribution, $\widehat\pi$
and the assumptions outlined in \cref{sec-assump} hold with $P_\epsilon = \widehat P$ and $\pi_\epsilon =\widehat\pi$.

Therefore, \cref{combinedthm,l1l2geomthm,l1l2geomthm-rev,cor:marg-error-bd,perturberrorthm,mse_perturb} can all be applied.
In particular, $\widehat P$ is $L_2(\pi)$-geometrically ergodic with factor $1-(\alpha-R (1+\normpi{Q}))$, it is a.e.-TV geometrically ergodic, and
\[
\normpi{\widehat \pi - \pi} \leq \frac{R (1+\normpi{Q})}{\sqrt{\alpha^2 - R^2 (1+\normpi{Q})^2}} \ ;
\]
and, if $\widehat P$ is reversible then it is $L_2(\widehat\pi)$ geometrically ergodic with factor $(1-(\alpha-R (1+\normpi{Q})))$.
\end{theorem}

The above theorem provides an alternative to the analogous
result of Corollary~2.3 from \citep{alquier2016noisy}, relaxing the
uniform ergodicity assumption.
In particular, it requires that $Q\in\mathcal{B}(L_2(\pi))$ and that
$R(1+\normpi{Q})<\alpha$. The first of these requirements is
not dramatically limiting since the user has control over the choice
of $Q$. The second of these requirements is also not
dramatically limiting as control over $R$ may be interpreted as
limiting the amount of noise in the nMH algorithm and such control is
required regardless in order to ensure the accuracy of approximation in
both the geometrically ergodic and uniformly ergodic cases.

\subsection{Application to Fixed Deterministic Approximations}
Suppose we run a fixed Metropolis--Hastings algorithm, but replace the target density with one which is close everywhere. Perhaps this alternative density is easier to compute (e.g. replacing an integral with a Laplace approximation as in \citet{kass1990validity}, or replacing a full sample with a coreset for sub-sampled Bayesian Inference as in \citet{campbell2019automated}). By construction we would know that the approximate target distribution is close to the ideal target distribution. The question still remains whether geometric ergodicity is preserved. We resolve this question in the case that the approximation has constant relative error.

\begin{corollary}
  \label{cor:deterministic-approx}
  Suppose we can approximate the unnormalized target density, $C\pi$, by $\widehat \pi$, with a $\theta$-bounded relative error;
  \[
    \sup_{x\in\Xx}\abs{\log\frac{C\pi(x)}{\widehat\pi(x)}} \leq \theta \ .
  \]

  Then if the Metropolis--Hastings algorithm with proposal kernel $Q$ is $L_2(\pi)$-geometrically ergodic with factor $(1-\alpha)$, and if $\theta < \frac{\alpha}{2(1+\normpi{Q})}$, then the corresponding approximate transition kernel, $\hat P$, is $L_2(\widehat \pi)$-geometrically ergodic and
  \[
  \normpi{\widehat \pi - \pi} \leq \frac{2\theta (1+\normpi{Q})}{\sqrt{\alpha^2 - 4\theta^2 (1+\normpi{Q})^2}} \ ;
  \]
\end{corollary}
\begin{proof}
  Since the function $x\mapsto 1\wedge \exp(x)$ is $1$-Lipschitz, we have:
  \[
    \abs{r(y\vert x)}
      & = \abs{a(y\vert x) - \widehat a(y\vert x)} \\
      & \leq \abs{\log\frac{\pi(y)q(x\vert y)}{\pi(x)q(y\vert x)} - \log\frac{\widehat \pi(y)q(x\vert y)}{\widehat\pi(x)q(y\vert x)} } \\
      & = \abs{\log\frac{C\pi(y)}{\widehat \pi(y)} - \log\frac{C\pi(x)}{\widehat\pi(x) }} \\
      & \leq 2 \theta
  \]
  So, $\widehat P$ will be $L_2(\pi)$-geometrically ergodic as long as $P$ was geometrically ergodic with some factor $0\leq(1-\alpha)<1$ and
  \[\label{eqn-smallrelerr}
    \theta < \frac{\alpha}{2(1+\normpi{Q})} \ .
  \]
  Moreover, in this case, $\widehat P$ is reversible. Thus, we can use \cref{l1l2geomthm-rev} to obtain $L_2(\widehat \pi)$-geometric ergodicity of $\widehat P$, with factor $1- \alpha + 2\theta(1+\normpi{Q})$.
\end{proof}
In this scenario, we can also use \cref{mse_perturb} to get quantitative bounds for the mean-squared error of any Monte Carlo estimates made using $\widehat P$, or any of our other results in \cref{combinedthm,l1l2geomthm,l1l2geomthm-rev,cor:marg-error-bd,perturberrorthm} as needed.

\begin{example}[Independence Sampler]
  The previous result also immediately gives that if $\rnderiv{\widehat \pi}{\pi}$ is bounded above by $C<\exp(1/4)$ and below by $c>\exp(-1/4)$ then the independence sampler for $\widehat \pi$ with proposals from $\pi$ is geometrically ergodic with factor $4\max(\log C, -\log(c))$.
  This is, however, sub-optimal when compared to \citet{smith1996exact} which only requires a finite upper bound on $\rnderiv{\widehat \pi}{\pi}$ to establish \emph{uniform ergodicity}.
\end{example}

\begin{example}[Laplace Approximation for GLMMs]
   Generalized linear mixed models (GLMMs) (see \citet{breslow1993approximate},\citet{mcculloch2005generalized}, etc.) are widely used in the modelling of non-normal response variables under repeated or correlated measurements. They are the natural common extension of generalized linear models and linear mixed effects models.  They handle dependence between observations by introducing Gaussian latent variables. These \emph{random effects} are nuisance variables for the purpose of inference. In order to perform Bayesian inference for GLMMs, one requires samples from the marginal posterior distribution of the parameters given the data. The marginal posterior, here, is the posterior for the parameters given the observations, in contrast to the joint posterior of the random effects and the parameters given the data.

   This can be approached in two ways. One option is to obtain samples for the random effects and parameters jointly given the data, and discard the random effects to get marginal posterior samples for the parameters. The second option is to approximate the likelihood by integrating (numerically) over the random effects, and using the resulting approximate likelihood in the calculations involving the unnormalized posterior for the parameters.

   In the second case, when the prior for the parameters is compactly supported, if one had established a result saying that a particular MH procedure for the exact posterior distribution of the parameters would be geometrically ergodic, then one could directly transfer this result to the approximate posterior computed using a Laplace approximation, at least for large enough samples. This is valid since the Laplace approximation has constant relative error on compact sets, and the relative error decreases with sample size (see \citet{tierney1986accurate}). Hence, for a large enough sample size \cref{eqn-smallrelerr} will be satisfied regardless of what the proposal kernel $Q$ was (as long as $\normpi{Q}$ was finite).
\end{example}

\begin{example}[Uniform Coresets]
  In Bayesian inference with large samples, an approach to reducing the computational burden of evaluating the likelihood in the unnormalized posterior for MCMC accept/reject steps is to select a representative subsample of the data and to up-weight the contributions of each of the selected samples in a way to best approximate the original likelihood. These up-weighted subsamples are called \emph{coresets}. They naturally give rise to approximate MCMC methods in which the true posterior is replaced by an approximation based upon a coreset. Several methods for coreset construction exist, however relatively little work has been done to assess their impact upon approximate MCMC methods. We will consider the \emph{uniform coreset} construction of \citet{huggins2016coresets} (as so named in \citep{campbell2019automated}).

  \citet[Theorem~3.2]{campbell2019automated} provides the guarantee that, with probability $(1-\delta)$, the unnormalized approximate posterior $\hat C \hat \pi$ based on a uniform coreset of size $M$ will satisfy
    \[
        \sup_{x\in \Xx}\frac{1}{\abs{\Ll(x)}}\abs{\log \frac{\hat C \hat \pi(x)}{C\pi(x)}}
            \leq \frac{\sigma }{\sqrt{M}} \rbra{\frac{3}{2}D +\ol \eta\sqrt{2\log\rbra{1/\delta}}}
    \]
    where $\sigma = \sum_{n=1}^N \sigma_n$, $N$ is the number of observations, $\sigma_n = \sup_{x\in \Xx}\abs{\frac{\Ll_i(x)}{\Ll(x)}}$, $\Ll_i(x)$ is the log-likelihood of parameter $x$ at the $i$th observation, $\Ll(x) = \sum_{i=1}^N \Ll_i(x)$ is the log-likelihood of the dataset
    \[
    \ol \eta = \max_{i,j\in\set{1,\dots,N}} \sup_{x\in \Xx}\frac{1}{\abs{\Ll(x)}}\abs{\frac{\Ll_i(x)}{\sigma_i}  - \frac{\Ll_j(x)}{\sigma_j}},
    \]
    and $D$ is the \emph{approximate dimension} of $\set{\Ll_i}_{i=1}^n$ (\cite[Definition~3.1]{campbell2019automated})

    If in addition to assuming that $\set{\sigma_i}_{i=1}^N$ are all finite as in \citep[Section~3]{campbell2019automated}, one were to assume that $\abs{\Ll(x)}$ is bounded as a function of $x$,
    then the uniform coreset result would imply the conditions of our \cref{cor:deterministic-approx}, namely that
      \[
        \sup_{x\in\Xx}\abs{\log\frac{C\pi(x)}{\widehat\pi(x)}}
          & \leq \frac{\sigma \norm{\Ll}_\infty}{\sqrt{M}} \rbra{\frac{3}{2}D +\ol \eta\sqrt{2\log\rbra{1/\delta}}},
        \]
    with high probability.
    Consequently, for any proposal kernel $Q:L_2(\pi)\to L_2(\pi)$ we should be able to choose $M$ sufficiently large so that with high probability
    \[
      \frac{\sigma \norm{\Ll}_\infty}{\sqrt{M}} \rbra{\frac{3}{2}D +\ol \eta\sqrt{2\log\rbra{1/\delta}}} < \frac{\alpha}{2(1+\normpi{Q})}.
    \]
    Hence the approximating Markov chain will by geometrically ergodic with high probability.
\end{example}

\subsection{Application to Monte Carlo Within Metropolis}
Following \citet{medina2019perturbation}, we can get bounds for the simple Monte Carlo within Metropolis algorithm (MCwM).
This is the special case of nMH where we approximate the likelihood ratio $\frac{\pi(y)}{\pi(x)} = \frac{\EE\Pi(y,Z)}{\EE\Pi(x,Z)}$ by $\widehat{\lcr({\frac{\pi(y)}{\pi(x)}})} = \frac{\sum_{i=1}^N \Pi(y,Z_{i})}{\sum_{i=N+1}^{2N} \Pi(x,Z_{i})}$
using a new independent sample taken each time the likelihood is evaluated. In the notation of the previous section,
\[
  \widehat a(y\vert x,z) = 1 \wedge \frac{q(x\vert y)\sum_{i=1}^N \Pi(y,z_{i})}{q(y\vert x\sum_{i=N+1}^{2N} \Pi(x,z_{i})}
\]

Let
\[
  W_k(x)
    & = \frac{1}{k \pi(x)} \sum_{i=1}^k \Pi(x,Z_{i}) \\
  i_{k}(x)^2
    & = \EE[W_k(x)^{-2}]\\
  s(x)
    & = \frac{1}{\sqrt{\pi(x)}} \sdev(\Pi(x,Z_1)) \\
\]
\citep[Lemma~14]{medina2019perturbation} tells us that if there is a $k\in \NN$ such that $i_{k}(x)<\infty$ for all $x\in\Xx$ then for $N\geq k$
\[
  \abs{r(y\vert x)}
    & \leq a(y\vert x) \frac{1}{\sqrt{N}} i_k(y)\lcr({s(x)+s(y)}) \\
    &\leq \frac{1}{\sqrt{N}} i_k(y)\lcr({s(x)+s(y)})
\]

\begin{corollary}
  Let $P$ be the Metropolis--Hastings transition kernel for the target density $\pi$ and proposal kernel $Q$.
  Let $\widehat P_N$ be the corresponding MCwM transition kernel when $\pi(\cdot)$ is approximated by $\frac{1}{N} \sum_{i=1}^N \Pi(\cdot,Z_{i})$.

  Assume that $s$ and $i_k$ as defined above are uniformly bounded for some $k\in\Nats$. Suppose further that $N_0= \max\lcr({k, \frac{4 \norm{i_k}_\infty^2 \norm{s}_\infty^2 (1+\normpi{Q})^2}{\alpha^2}})$, and $N\geq \floor{N_0}+1$.

  Then $\widehat P_N$ is reversible and $L_2(\pi)$-geometrically ergodic with factor $  1- \alpha + \frac{1}{\sqrt{N/N_0}}$, and has a stationary distribution, $\widehat \pi_N(x) \propto  \frac{\pi(x)}{N\ \EE \lcr[{\lcr({\sum_{i=1}^N \Pi(x,Z_i)})^{-1}}]}$ with
  \[
    \normpi{\pi - \widehat \pi_N}\leq \sqrt{\frac{N_0}{N\alpha^2 -N_0}}
  \]
\end{corollary}
\begin{proof}
Suppose that $N\geq \floor{N_0}+1$.
From \cref{combinedthm}, we know that the perturbed chain, $\widehat P_N$ is $L_2(\pi)$-geometrically ergodic with factor $  1- \alpha + \frac{1}{\sqrt{N/N_0}}$, has a stationary distribution, $\widehat\pi_N$ with
\[
  \normpi{\pi - \widehat \pi_N}\leq \sqrt{\frac{N_0}{N\alpha^2 -N_0}} \ .
\]
Moreover, by inspection, $\widehat P_N$ is reversibility with respect to $\widehat \pi_N(x) \propto  \frac{\pi(x)}{N\ \EE \lcr[{\lcr({\sum_{i=1}^N \Pi(x,Z_i)})^{-1}}]}$.
Thus, we can use \cref{l1l2geomthm-rev} to obtain $L_2(\widehat \pi_N)$-geometric ergodicity of $\widehat P_N$, with factor $1- \alpha + \frac{1}{\sqrt{N/N_0}}$.
\end{proof}

\begin{remark}
  A simple scenario under which these $i_k$ and $s$ are uniformly bounded is when the joint density of $x$ and $Z$ is bounded above an below by a multiple of the marginal of $x$, so that
    \[
        \frac{\Pi(x,z)}{\pi(x)} \in [c,C]
    \]
    for all $(x,z)\in\Xx\times\Zz$. This condition is essentially tight if we wish to take $k=1$ and the base measure to be the Lebesgue measure restricted to $U\subset\Reals^d$; in this case the condition $\norm{i_k(x)}_{L_\infty}<\infty$ implies that
    \[
        \int_{U} \frac{\pi(x)}{\Pi(x,z)} dz = \EE_{Z\sim\frac{\Pi(x,\cdot)}{\pi(x)}} \frac{\pi(x)^2}{\Pi(x,\cdot)^2} <\infty
    \]
    for all $x$. That is, the reciprocal of the conditional density of $Z$ given $X=x$ has a finite integral for each $x$.
\end{remark}
\begin{remark}
    More generally, \citep[Lemma~23]{medina2019perturbation} tells us that if $\EE[W_{k_0}(x)^{-p}]<\infty$ for some $k_0\in\Nats$ and $p>0$ then for $k\geq k_0 \ceil{\frac{2}{p}}$, $i_{k}(x)^2 < \EE[W_{k_0}(x)^{-p}]$.
    Therefore, in order to uniformly bound $i_k(x)$, it is sufficient to bound $\EE[W_{k_0}(x)^{-p}]$ uniformly in $x$ for some $k_0\in\Nats$, $p>0$. This is much less restrictive than trying to bound $i_1(x)$.
    In the case that $p<1, k_0=1$ this is much less restrictive then $p=2,k_0=1$; it is equivalent to requiring that tempered versions of conditional distribution $\frac{\Pi(x,\cdot)}{\pi(x)}$ can be normalized by uniformly bounded normalizing constants.
    This would be true, if for example $(Z| X=x)\sim\Nn(\mu(x),\sigma^2(x))$ with $\sigma^2(x)$ uniformly bounded in $x$. More generally, using $0<p<1$, instead of $p=2$ whenever the conditional law of $Z$ has uniform exp-poly tails, $\frac{\Pi(x,z)}{\pi(x)}\leq \exp( -C\abs{z-\mu(x)}^\alpha)$, with $\alpha>0$, the $p$-version of the condition would hold.

\end{remark}

We could also use \cref{mse_perturb} to get quantitative bounds for the mean-squared error of any Monte Carlo estimates made using $\widehat P_N$, or any of our other results in \cref{combinedthm,l1l2geomthm,l1l2geomthm-rev,cor:marg-error-bd,perturberrorthm} as needed.

In \citep{medina2019perturbation}, they also consider a case where the the assumption that $s$ and $i_k$ are uniformly bounded is dropped, and instead, the perturbed kernel is restricted to a bounded region. We do not address this case here.

%% file: section-files/geomPerturb_apdx-perturbbounds-proofs.tex
\subsection{Proof of \cref{combinedthm}}
\label{ax:proof-combinedthm}

The following lemma is contained in the remark after Theorem~2.1 of
\citep{roberts1997geometric}; we prove it here as well since the proof is so simple.

\begin{lemma}[Remark in \citep{roberts1997geometric}]\label{normeqlemma}
For any probability measure $\mu\in L_2(\pi)$,
\*[
    \normpi{ \mu - \pi}^2 = \normpi{\mu}^2 -1
\]
\end{lemma}
\begin{proof}
\*[
0\leq\normpi{\mu-\pi}^2
    &=\int \lcr({\rnderiv\mu\pi -1})^2 \dee \pi
    =\int \lcr({\lcr({\rnderiv\mu\pi})^2 -2 \rnderiv\mu\pi +1 }) \dee\pi\\
    &=\int \lcr({\rnderiv\mu\pi})^2\dee\pi-2\int \dee\mu +\int \dee\pi
    =\normpi\mu^2-1
\]
\end{proof}

We will make use of the following simplified version of Theorem~2.1 from \citep{roberts1997geometric} as well:

\begin{proposition}[Equivalent definitions of $L_2(\pi)$ geometric ergodicity from \citep{roberts1997geometric}]
\label{rr97thm2}
For a reversible Markov chain with kernel $P$ and stationary distribution $\pi$ on state space $\Xx$, the following are equivalent (and $\rho$ is equal in both cases):
\begin{itemize}
\item[(i)] $P$ is $L_2(\pi)$-geometrically ergodic with optimal rate $\rho$ and coefficient function $C(\mu) = \normpi{\mu-\pi}$,
\item[(ii)] $P$ has $L_{2,0}(\pi)$-spectral radius and norm both equal to $\rho$;
\*[
	\sup_{\nu \in L_{2,0}(\pi)\setminus\set{0}}\frac{\normpi{\nu P}}{\normpi{\nu}} = \rho = r(P\restrict{L_{2,0}(\pi)}) \ ,
\]
Where
\[
r(P\restrict{L_{2,0}(\pi)}):=\sup\set{\abs{\rho}: \rho\in\CC \andT \rbra{P\restrict{L_{2,0}(\pi)}-\rho I_{L_{2,0}(\pi)}} \text{ is not invertible}}
\]
\end{itemize}
\end{proposition}
Note that while when the kernel is reversible we may take $C(\mu) = \normpi{\mu-\pi}$ in the bound corresponding $L_2(\pi)$-GE with optimal rate $\rho$, this is not true for non-reversible chains.
By applying the above theorem in our context we have:

\begin{lemma}
\label{normprop}%
Under the assumptions of \cref{sec-assump},
\*[
    \normpi{ \nu_1 P^n - \nu_2 P^n}\leq (1-\alpha)^n\normpi{\nu_1-\nu_2}
\]
for any probability distributions $\nu_1,\nu_2\in L_2(\pi)$.
In particular, taking $\nu_2=\pi$,
\*[
    \normpi{\nu_1 P^n - \pi} \leq (1-\alpha)^n \normpi{\nu_1-\pi}
=(1-\alpha)^n \sqrt{\normpi{ \nu_1} ^2-1}
\]
and applying Cauchy-Schwarz yields
\*[
    \normpiO{\nu_1 P^n - \pi}\leq \normpi{ \nu_1 P^n - \pi} \leq
(1-\alpha)^n \normpi{ \nu_1 - \pi}
\]
\end{lemma}

We begin with a first result giving sufficient conditions under which the
stationary distribution $\pi_\epsilon$ of the perturbed chain is
in $L_2(\pi)$:

\begin{lemma}
\label{lem:piep}
Under the assumptions of \cref{sec-assump},
if in addition $\epsilon<\alpha$, then $P_\epsilon$ has a unique stationary distribution, $\pi_\epsilon\in L_2(\pi)$, and $\normpi{\pi_\epsilon-\pi}\leq \frac{\epsilon}{\alpha-\epsilon}$.
\end{lemma}

\begin{proof}
Since $P_\epsilon$ is $\pi$-irreducible and aperiodic, it has at most one stationary distribution, $\pi_\epsilon$, with $\pi_\epsilon\ll \pi$ (see for example \citep[Corollary~9.2.16]{douc2018markov}).

Suppose for now that $\pi P_\epsilon^n$ has an $L_2(\pi)$ limit, $\pi_\epsilon$;
Then, using the triangle inequality, and the contraction property ($\normTV{P_\epsilon} = 1$), and Cauchy-Schwarz
\*[
	\normTV{\pi_\epsilon P_\epsilon - \pi_\epsilon}
		& \leq \normTV{\pi_\epsilon P_\epsilon - \pi P_\epsilon^{n}}
			+ \normTV{\pi P_\epsilon^{n} - \pi_\epsilon}\\
		& \leq \normTV{\pi_\epsilon - \pi P_\epsilon^{n-1}}
			+ \normTV{\pi P_\epsilon^{n} - \pi_\epsilon}\\
			& \leq \normpi{\pi_\epsilon - \pi P_\epsilon^{n-1}}
				+ \normpi{\pi P_\epsilon^{n} - \pi_\epsilon}
			\stkm\to{n\to\infty} 0
\]
we find that $\pi_\epsilon$ must be stationary for $P_\epsilon$.

It remains to verify that $\{\pi P_\epsilon^n\}_{n\in\NN}$ is an
$L_2(\pi)$-Cauchy sequence, and thus from completeness it must have
an $L_2(\pi)$-limit. To this end, define $Q_\epsilon = (P_\epsilon - P)$. Let $\mathbf{2}^k=\{0,1\}^k$ for all $k\in\NN$. We will expand $\pi(P+Q_\epsilon)^n$ and use the following facts:
\begin{enumerate}
\item[(A)] $\forall R\in\Bb(L_2(\pi))\ [\pi P^n R = \pi R]$
\item[(B)] $Q_\epsilon : L_2(\pi)\to L_{2,0}(\pi)$
\item[(C)] $P\restrict{L_{2,0}(\pi)}\in\Bb(L_{2,0}(\pi))$ and $\normpiR{P\restrict{L_{2,0}(\pi)}}\leq (1-\alpha)$
\end{enumerate}

Since the operators $P$ and $Q_\epsilon$ do not (necessarily) commute, when we expand $(P+Q)^n$ we must have one distinct term per binary sequence of length $n$. We can then group terms by the number of leading $P$s, and use (A) to cancel the leading terms.

Let $m,n\in\mathbf{N}$ be arbitrary with $m\leq n$.

\*[
&\hspace{-1em}\normpi{ \pi P_\epsilon^n -\pi P_\epsilon^m}\\
    &= \normpi{ \pi (P+Q_\epsilon)^n - \pi (P+Q_\epsilon)^m} \\
    &= \normpi{\pi \lcr[{\lcr({\sum_{\mathbf{b}\in\mathbf{2}^n}\prod_{j=1}^n P^{b_j} Q_\epsilon^{1-b_j} })-\lcr({\sum_{\mathbf{b}\in\mathbf{2}^m} \prod_{j=1}^m  P^{b_j} Q_\epsilon^{1-b_j} })}]}\\
    & =\lcr\Vert{\pi\lcr[{\lcr({P^n+\sum_{k=0}^{n-1}P^{n-k-1}Q_\epsilon\sum_{\mathbf{b}\in\mathbf{2}^{k}}\prod_{j=1}^{k} P^{b_j} Q_\epsilon^{1-b_j}}) - \lcr({P^m+\sum_{k=0}^{m-1}P^{m-k-1}Q_\epsilon\sum_{\mathbf{b}\in\mathbf{2}^{k}}    \prod_{j=1}^{k} P^{b_j} Q_\epsilon^{1-b_j} })}]}\Vert_{L_2(\pi)}\\
    & = \lcr\Vert{\lcr({\pi+\sum_{k=0}^{n-1}\pi Q_\epsilon\sum_{\mathbf{b}\in\mathbf{2}^{k}} \prod_{j=1}^{k} P^{b_j} Q_\epsilon^{1-b_j} }) -\lcr({\pi+\sum_{k=0}^{m-1}\pi Q_\epsilon\sum_{\mathbf{b}\in\mathbf{2}^{k}} \prod_{j=1}^{k} P^{b_j} Q_\epsilon^{1-b_j} })}\Vert_{L_2(\pi)}\\
    & = \normpi{\pi \sum_{k=m}^{n-1}Q_\epsilon\sum_{\mathbf{b}\in\mathbf{2}^{k}} \prod_{j=1}^{k} P^{b_j} Q_\epsilon^{1-b_j}}\\
    &\leq \epsilon
    \sum_{k=m}^{n-1}\sum_{\mathbf{b}\in\mathbf{2}^{k}} \prod_{j=1}^{k}
    (1-\alpha)^{b_j} \epsilon^{1-b_j}\\
    &=\epsilon \sum_{k=m}^{n-1}(1-\alpha+\epsilon)^{k}\\
    &\leq \frac{\epsilon}{\alpha-\epsilon} (1-\alpha+\epsilon)^m
\]
Since this upper bound on $\normpi{\pi P_\epsilon^n -  \pi
P_\epsilon^m}$ decreases to 0 monotonically in $m=\min(m,n)$
then the sequence must be $L_2(\pi)$-Cauchy.

Now, to bound the norm of $\pi_\epsilon$ we take $m=0$ and we get that for all $n\in \Nats$:
\*[
	\normpi{ \pi P_\epsilon^n -\pi} \leq \frac{\epsilon}{\alpha-\epsilon}
\]
From the continuity of norm, it must be the case that $\normpi{\pi_\epsilon- \pi}\leq \frac{\epsilon}{\alpha-\epsilon} $
\end{proof}

\begin{lemma}
\label{lem:l2bound}%
Under the assumptions of \cref{sec-assump},
if in addition $\epsilon<\alpha$ then
\*[1\leq\normpi{\pi_\epsilon} \leq
\frac{\alpha}{\sqrt{\alpha^2-\epsilon^2}}\]
and
\*[0\leq\normpi{\pi-\pi_\epsilon} \leq
\frac{\epsilon}{\sqrt{\alpha^2-\epsilon^2}}.\]
\end{lemma}

\begin{proof}
The two lower bounds are immediate from
\cref{normeqlemma} and the positivity of norms:
\*[
0 \leq \normpi{\pi-\pi_\epsilon}^2 = \normpi{\pi_\epsilon}^2-1
\]
To derive the first upper bound, we apply
\cref{normeqlemma}, our assumptions about the operators $P$
and $P_\epsilon$, and triangle inequality, to $\Vert \pi-\pi_\epsilon\Vert_2$:
\*[
\sqrt{\normpi{\pi_\epsilon}^2-1}= \normpi{\pi-\pi_\epsilon}
    &=\normpi{\pi P - \pi_\epsilon P+ \pi_\epsilon
    P - \pi_\epsilon P_\epsilon}\\
    &\leq\normpi{\pi P - \pi_\epsilon P}  + \normpi{ \pi_\epsilon
    P - \pi_\epsilon P_\epsilon} \\
    &\leq (1-\alpha)\normpi{\pi-\pi_\epsilon} +\epsilon\normpi{
    \pi_\epsilon} \\
    &= (1-\alpha)\sqrt{\normpi{ \pi_\epsilon} ^2-1}+\epsilon\normpi{
    \pi_\epsilon}
\]
Collecting the square roots and squaring both sides yields
\*[
\alpha^2\lcr({\normpi{ \pi_\epsilon} ^2-1})
    &\leq \epsilon^2\normpi{\pi_\epsilon}^2
\]
which implies that
\*[
\normpi{\pi_\epsilon}^2
    &\leq \frac{\alpha^2}{\alpha^2-\epsilon^2}
\]
Finally, the second upper bound is derived from the first one,
again using \cref{normeqlemma}:
\*[
\normpi{\pi-\pi_\epsilon}^2 = \normpi{\pi_\epsilon}^2-1 \leq \frac{\alpha^2}{\alpha^2-\epsilon^2} - 1 = \frac{\epsilon^2}{\alpha^2-\epsilon^2}
\]
\end{proof}

We next observe that our assumptions imply that for small
enough perturbations, the perturbed chain $P_\epsilon$ is geometrically
ergodic in the $L_2(\pi)$ norm.

\begin{lemma}
\label{lem:l2pi}
Under the assumptions of \cref{sec-assump},
if $\epsilon<\alpha$, then $P_\epsilon$ is $L_2(\pi)$-geometrically
ergodic, with factor $\leq 1-(\alpha-\epsilon)$.
\end{lemma}

\begin{proof}
Suppose that $\nu\in L_{2,0}(\pi)$.  Then
\*[
	\normpi{\nu P_\epsilon}
		& \leq \normpi{\nu (P_\epsilon -P)} +\normpi{\nu P}\\
		& \leq \epsilon \normpi{\nu} + (1-\alpha) \normpi{\nu}\\
		& =(1-\alpha+\epsilon)\normpi{\nu}
\, .
\]
Thus, for any probability measure $\mu\in L_2(\pi)$,
since $\pi_\epsilon \in L_2(\pi)$ we have
\*[
	\normpi{\mu P_\epsilon^n -\pi_\epsilon}
		& = \normpi{(\mu-\pi_\epsilon) P_\epsilon^n} \\
		& \leq (1-(\alpha-\epsilon))^n \normpi{\mu -
\pi_\epsilon}
\, .
\]
\end{proof}

Combining \cref{lem:piep,lem:l2bound,lem:l2pi} together with the triangle inequality immediately yields \cref{combinedthm}.

\subsection{Proofs of \cref{l1l2geomthm}, \cref{l1l2geomthm-rev} and \cref{cor:marg-error-bd}}
\label{ax:proofs-l1l2-margerr}

\begin{definition}\normalfont\label{hypersmall}
Following \citep{roberts1997geometric}, a subset $S\subset\Xx$
is called \textit{hyper-small} for the $\pi$-irreducible Markov kernel $P$ with stationary measure $\pi$ if $\pi(S)>0$
and there exists $\delta_S>0$ and $k\in\NN$ such that $\rnderiv{
P^k(x,\cdot)}{\pi} (y)\geq \delta_S\one_S(x)\one_S(y)$ or equivalently
$P^k(x,A)\geq \delta_S\pi(A)$ for all $x\in S$ and $A\subset S$
measurable.
\end{definition}

Lemma~4 of \citet{jain1967contributions}
states that on a countably generated state
space (as we have assumed herein), every set of positive
$\pi$-measure contains a hyper-small subset.

\begin{lemma}[Existence of Hyper-Small Subsets from \citep{jain1967contributions}]
\label{jjresult}
Suppose that $(\Xx,\Sigma)$ is countably generated. Suppose that $X$ is a a $\phi$-irreducible Markov chain on $\Xx$ with kernel $P$ for some $\sigma$-finite measure $\phi$ on $\Xx$. Then any set $K\subset \Xx$ with $\phi(K)>0$ contains a set $S_K$ such that (for some $n_{K}\in\NN$)
\*[
	\inf_{(x,y) \in S_K\times S_K} \rnderiv{P^{n_K}(x,\cdot)}{\pi} (y) = \delta >0
\]
\end{lemma}
In the case that a stationary distribution, $\pi$, for $P$ exists, without loss of generality we can take $\phi=\pi$. In this case, it is immediate that any set $(S_K, n_K)$ satisfying \cref{jjresult} also satisfies \cref{hypersmall}.

Also of importance to us is the following variant of Proposition~2.1 of \citep{roberts1997geometric}, which provides a characterization of geometric
ergodicity in terms of convergence to a hyper-small set.

\begin{proposition}[Equivalent characterizations of $\pi$-a.e.-TV geometric ergodicity from \citep{roberts1997geometric} and \citet{nummelin1978geometric}]
\label{ge-char}
Suppose that $(\Omega,\Sigma)$ is countably generated, and that $X$ is a a $\phi$-irreducible Markov chain on $\Xx$ with kernel $P$ with stationary distribution $\pi$. Then the following are equivalent:
\begin{itemize}
	\item[(i)] There exists $\rho_\tv\in(0,1)$ such that $P$ is $\pi$-a.e.-TV geometrically ergodic with factor $\rho_\tv$
	\item[(i${}^{\prime\prime}$)] There exists a hyper-small set $S\subset \Xx$, and constants $\rho_S<1$, $C_S\in\RR_+$ such that:
	\*[
		\normTV{\int \frac{\one_S(y)\pi(\dee y)}{\pi(S)} P^n(y,\cdot) -\pi} \leq C_S \rho_S^n	 \qquad \forall n\in\NN
	\]
	\item[(ii)] There exists a $\pi$-a.e. finite, measurable function $V:\Xx\to[1,\infty]$ with $\pi(V^2) <\infty$, and $\rho_V\in(0,1)$, and $C>0$ such that:
	\*[
		2\normTV{\delta_x P^n -\pi}\leq \norm{\delta_x P^n -\pi}_V \leq C V(x) \rho_V^n
	\]
	where $\norm{\mu}_V = \sup_{\abs{f}\leq V} \abs{\mu(f)}$.
\end{itemize}
\end{proposition}

\begin{proof}[Proof of \cref{l1l2geomthm}]
	\textit{(i)} Let $S$ be a hyper-small set for $P_\epsilon$ (which exists from \cref{jjresult}, since $P_\epsilon$ is $\pi_\epsilon$-irreducible). Then the measure $\mu_S$ defined by $\rnderiv{\mu_S}{\pi} = \frac{\Ind{S}}{\pi_\epsilon(S)} \rnderiv{\pi_\epsilon}{\pi}$ has (by H\"older's inequality, and since $\pi_\epsilon\in L_2(\pi)$) that $\normpi{\mu_S}^2 \leq \normpi{\pi_\epsilon}^2 \pi_\epsilon(S)^{-2}<\infty $, and hence $\mu_S\in L_2(\pi)$. Then (by Cauchy-Shwarz again):
	\*[
		\normTV{\int \frac{\one_S(y)\pi_\epsilon(\dee y)}{\pi_\epsilon(S)} P_\epsilon^n(y,\cdot) -\pi_\epsilon}
			\leq \frac{1}{2} \normpi{\mu_S P_\epsilon^n -\pi_\epsilon} \leq \normpi{\mu_S - \pi_\epsilon} (1-\alpha +\epsilon)^n
	\]
	which, along with \cref{ge-char}, establishes that $P_\epsilon$ is $\pi_\epsilon$-a.e.-TV geometrically ergodic with some factor $\rho_\tv\in(0,1)$.

	\textit{(ii)} Suppose that $\mu\in L_\infty(\pi_\epsilon)$. Then $\mu \in L_2(\pi)$ since $\rnderiv{\mu}{\pi} \leq \norm{\mu}_{L_\infty(\pi_\epsilon)} \rnderiv{\pi_\epsilon}{\pi}$. Since $\mu P_\epsilon^n - \pi_\epsilon \in L_1(\pi_\epsilon)\subset L_1(\pi)$ then
	\[
		\norm{\mu P_\epsilon^n - \pi_\epsilon}_{L_1(\pi_\epsilon)}
			= \norm{\mu P_\epsilon^n - \pi_\epsilon}_{L_1(\pi)}
				= 2\normTV{\mu P_\epsilon^n - \pi_\epsilon}.
	\]
	Applying this equality as well as Cauchy-Schwarz we get
	\[
		\norm{\mu P_\epsilon^n - \pi_\epsilon}_{L_1(\pi_\epsilon)}
			& = \norm{\mu P_\epsilon^n - \pi_\epsilon}_{L_1(\pi)} \\
			& \leq \norm{\mu P_\epsilon^n - \pi_\epsilon}_{L_2(\pi)} \\
			& \leq \normpi{\mu-\pi_\epsilon}(1-\alpha-\epsilon)^n
	\]

	\textit{(iii)}
	If $\pi\in L_\infty(\pi_\epsilon)$ and $\mu\in L_2(\pi_\epsilon)$ then
	\[\label{eqn:eps-contraction}
		\norm{\mu P_\epsilon^n - \pi_\epsilon}_{L_2(\pi_\epsilon)}^2
			& = \int \lcr({\rnderiv{\mu P_\epsilon^n - \pi_\epsilon}{\pi_\epsilon}})^2 d \pi_\epsilon \\
			& = \int \lcr({\rnderiv{\mu P_\epsilon^n - \pi_\epsilon}{\pi}})^2 \rnderiv{\pi}{\pi_\epsilon} d \pi\\
			&\leq \norm{\pi}_{L_\infty(\pi_\epsilon)}\int \lcr({\rnderiv{\mu P_\epsilon^n - \pi_\epsilon}{\pi}})^2  d \pi\\
			& = \norm{\pi}_{L_\infty(\pi_\epsilon)}\normpi{\mu P_\epsilon^n - \pi_\epsilon}^2\\
			& \leq \norm{\pi}_{L_\infty(\pi_\epsilon)}\normpi{\mu-\pi_\epsilon}^2(1-(\alpha-\epsilon))^{2n}
	\]

\end{proof}

\begin{proof}[Proof of \cref{l1l2geomthm-rev}]
  From \citet[Lemma~1]{baxter1995rates}, since $P_\epsilon$ has stationary measure $\pi_\epsilon$, then $P_\epsilon: L_2(\pi_\epsilon)\to L_2(\pi_\epsilon)$.
  Since $P_\epsilon$ is $(L_\infty(\pi_\epsilon), \norm{\cdot}_{L_1(\pi_\epsilon)})$-GE with factor $\rho_1 \leq (1-(\alpha-\epsilon))$ (as established by \cref{l1l2geomthm}) and $P_\epsilon$ is reversible, then it must also be $L_2(\pi_\epsilon)$-geometrically ergodic with factor $\rho =\rho_1$ by \cref{lem:rt-comment}.
\end{proof}

\begin{proof}[Proof of \cref{cor:marg-error-bd}]
Note that the assumption that $\normpiE{P-P_\epsilon}<\epseps$ implies $P-P_\epsilon : L_2(\pi_\epsilon)\to L_2(\pi_\epsilon)$.

\textit{(i)} Since $P_\epsilon$ is $L_2(\pi_\epsilon)$-geometrically ergodic with factor $(1-(\alpha-\epsilon))$ and $\pi_\epsilon$-reversible, we can reverse the roles of $P$ and $P_\epsilon$, so the result follows by \cref{combinedthm}.

\textit{(ii)}
Taking $\mu = \pi$ and $n=1$ in \cref{l1l2geomthm} \textit{(iii)},
\[
	\norm{\pi P_\epsilon - \pi_\epsilon}_{L_2(\pi_\epsilon)}^2
		& \leq \norm{\pi}_{L_\infty(\pi_\epsilon)}\normpi{\pi-\pi_\epsilon}^2(1-(\alpha-\epsilon))^2 \\
		&\leq \norm{\pi}_{L_\infty(\pi_\epsilon)}\frac{\epsilon^2}{\alpha^2-\epsilon^2}(1-(\alpha-\epsilon))^2
\]
Hence,
\[
		\norm{\pi - \pi_\epsilon}_{L_2(\pi_\epsilon)}
		&\leq \norm{\pi P - \pi P_\epsilon}_{L_2(\pi_\epsilon)} +\norm{\pi P_\epsilon - \pi_\epsilon}_{L_2(\pi_\epsilon)} \\
		& \leq \epseps\normpiE{\pi} + \norm{\pi}_{L_\infty(\pi_\epsilon)}^{1/2}\frac{\epsilon}{\sqrt{\alpha^2-\epsilon^2}}(1-(\alpha-\epsilon))\\
		& = \epseps\sqrt{\normpiE{\pi-\pi_\epsilon}^2+1} + \norm{\pi}_{L_\infty(\pi_\epsilon)}^{1/2}\frac{\epsilon}{\sqrt{\alpha^2-\epsilon^2}}(1-(\alpha-\epsilon))\\
		& \leq \epseps(\normpiE{\pi-\pi_\epsilon}+1) + \norm{\pi}_{L_\infty(\pi_\epsilon)}^{1/2}\frac{\epsilon}{\sqrt{\alpha^2-\epsilon^2}}(1-(\alpha-\epsilon))
\]
Hence,
\[\label{eqn:approx-error-eps}
	\normpiE{\pi-\pi_\epsilon}
		\leq \frac{\epseps + \norm{\pi}_{L_\infty(\pi_\epsilon)}^{1/2}\frac{\epsilon}{\sqrt{\alpha^2-\epsilon^2}}(1-(\alpha-\epsilon))}{1-\epseps}
\]
Finally,
\*[
	\normpiE{ \mu P_\epsilon^n -\pi}
		& \leq \normpiE{ \mu P_\epsilon^n -\pi_\epsilon}+\normpiE{\pi_\epsilon - \pi},
\]
The first term is bounded by \cref{l1l2geomthm}~\textit{(iii)}, and the second term is bounded by \cref{eqn:approx-error-eps}
\end{proof}

\subsection{Proofs of \cref{perturberrorthm} and \cref{mse_perturb}}

\subsubsection{Time-Averaging of Marginal Distributions}
\label{ax:time-avg-proofs}

\begin{proof}[Proof of \cref{perturberrorthm}]
The first result of  \cref{perturberrorthm} follows from the triangle inequality and \cref{combinedthm},
\*[
\normpi{\pi -\frac{1}{t}\sum_{k=0}^{t-1} \mu P_\epsilon^k}
    &\leq \frac{1}{t} \sum_{k=0}^{t-1}\normpi{\pi - \mu
    P_\epsilon^k }\\
    &\leq \frac{1}{t} \sum_{k=0}^{t-1}\left[(1-(\alpha-\epsilon))^k
    \normpi{\pi_\epsilon-\mu}
    +\frac{\epsilon}{\sqrt{\alpha^2-\epsilon^2}}\right]\\
    & \leq \frac{1-(1-(\alpha-\epsilon))^t}{t(\alpha-\epsilon)}\normpi{
    \pi_\epsilon-\mu}
    +\frac{\epsilon}{\sqrt{\alpha^2-\epsilon^2}} \ .
\]
The subsequent results follows from similarly via \cref{l1l2geomthm,l1l2geomthm-rev,cor:marg-error-bd}.
\end{proof}

\subsubsection{Covariance Bounds}
We turn our attention to the covariance structure of the original and perturbed chains.
There is an obvious isometric isomorphism between the space of measures $L_2(\pi)$ and the function space $L'_2(\pi) = \set{f:\Xx\to\RR \st \int f(x)^2\pi(dx) <\infty}$ equipped with the norm $\norm{f}^2_{L'_2(\pi)} = \int f(x)^2\pi(dx)$ where a measure $\mu$ is mapped to its Radon--Nikodym derivative $\mu\mapsto \rnderiv{\mu}{\pi}$.
For this reason, we need not distinguish between these spaces, and when dealing with a function $f\in L'_2(\pi)$ we may occasionally abuse notation and treat it as its associated measure.
Let $X_t$ and $X^\epsilon_t$ denote the original
and perturbed chains run from some initial measure $\mu\in L_2(\pi)$.

\begin{corollary}
\label{fgcor}%
Under the assumptions of \cref{sec-assump},
\begin{itemize}
\item[(a)] if $X_0\sim\pi$ (the initial distribution is the stationary distribution), then for $f,g\in L'_2(\pi)$
\[
	\Cov [f(X_t),g(X_s)]\leq (1-\alpha)^{|t-s|} \normpiD{f-\pi f}\normpiD{g-\pi g}\ ,
\]
\item[(b)] if $\epsilon<\alpha$, and $P_\epsilon$ is $\pi_\epsilon$-reversible, $\rho_2 = (1-(\alpha-\epsilon))$, and $X^\epsilon_0\sim\pi_\epsilon$ , then for $f,g\in L_2'(\pi_\epsilon)$
\[
	\Cov [f(X^\epsilon_t),g(X^\epsilon_s)]\leq \rho_2^{|t-s|}\normpiDE{ f - \pi_\epsilon f}\normpiDE{g - \pi_\epsilon g}\ ,
\]
\end{itemize}
where for a function $h:\Xx\to\RR$, $\pi h$ is the constant function equal to $\int h(s)\pi(ds)$ everywhere.
\end{corollary}

\begin{proof}
The proof of this result follows that of Corollary B.5 in \citep{johndrow2015approximations}.
We only show the proof for the original chain, however the proof for the perturbed chain is the same, since it is reversible and $L_2(\pi_\epsilon)$ geometrically ergodic with the appropriate factor, from \cref{l1l2geomthm-rev}.

Define the subspace
\*[
L'_{2,0}(\pi) = \{h\in L'_2(\pi): \int h(s)\pi(d s) = 0\}\ ,
\]
and the operator $F \in \mathcal{B}(L'_{2,0}(\pi))$ by
\*[
[F f](x) = \int P(x,d y)f(y) = \mathbb{E}[f(X_1)|X_0=x]
\]
From Lemma~12.6.4 of \citet{liu2008monte},
\*[
	\sup_{f,g\in L_2'(\pi)} \text{corr}(f(X_0),g(X_t))
			= \sup_{\substack{\normpiD{f}=1=\normpiD{g}\\f,g\in L'_{2,0}(\pi)}} \langle f,F^t g\rangle
			= \norm{F^t}_{L_{2,0}'(\pi)}
\]

Consider the canonical isomorphism between $L_2(\pi)$ and
$L'_2(\pi)$. The restriction of this isomorphism (on the right)
to elements of $L'_{2,0}(\pi)$ yields $L_{2,0}(\pi)$ (on the left) -- the signed measures with total measure $0$. The image of $F$ under the restricted isomorphism is the adjoint operator of $P$ restricted to $L_{2,0}(\pi)$. Since $P$ is $\pi$-reversible, it is self-adjoint, in $L_2(\pi)$ so $\norm{F}_{L_{2,0}'(\pi)}= \norm{P}_{L_{2,0}(\pi)}$.
\*[
	\norm{F^t}_{L_{2,0}'(\pi)} \leq \norm{F}_{L_{2,0}'(\pi)}^t = \norm{P\big\vert_{L_{2,0}(\pi)}}^t \leq(1-\alpha)^t
\]

Therefore
\*[
	\Cov(f(X_0),g(X_t))\leq \normpiD{f - \pi f}\normpiD{g-\pi g} (1-\alpha)^t
\]

Since $\Cov$ is symmetric, the shifted and symmetrized result holds for any $f,g\in L'_2(\pi)$:
\[
	\Cov [f(X_t),g(X_s)]\leq (1-\alpha)^{|t-s|} \normpiD{f -\pi f} \normpiD{ g - \pi g}
\]
\end{proof}
We present further bounds for the case that the initial distribution is not the stationary distribution in \cref{fgcor2}.
\begin{remark}\normalfont
\label{hstarremark}
Note in \cref{fgcor} that
\*[
\normpiD{h - \pi h} = \sqrt{\normpiD{ h}^2-(\pi h)^2}\leq \normpiD{h} \ .
\]
Also note that \[\normpiD{h} \leq \normpiD{h-\pi(h)}  + |\pi(h)| \ .\]
\end{remark}

\begin{corollary}
\label{fgcor2}
Under the assumptions of \cref{sec-assump},
\begin{itemize}
\item[(a)] if $X_0\sim\mu$, then for $f,g\in L'_4(\pi)$
\*[
&\Cov(f(X_t),g(X_{t+s}))\\
    &\qquad\leq (1-\alpha)^s \normpiD{f -\pi f}\normpiD{ g - \pi g} \\
        &\qquad\qquad + 2^{3/2}(1-\alpha)^{t+s/2}\normpi{\mu-\pi} \norm{f -\pi f}_{L_4'(\pi)}\norm{ g - \pi g}_{L_4'(\pi)}\\
        &\qquad\qquad -(\mu P^t f-\pi f)\left(\mu P^{t+s} g -\pi g\right)
\]
\item[(b)] if $\epsilon<\alpha$, and $P_\epsilon$ is $\pi_\epsilon$-reversible, $\rho_2 = (1-(\alpha -\epsilon))$, and $X^\epsilon_0\sim\mu$ , then for $f,g\in L'_4(\pi_\epsilon)$
\*[
&\Cov(f(X^\epsilon_t),g(X^\epsilon_{t+s}))\\
    &\qquad\leq \rho_2^s \normpiDE{f -\pi_\epsilon f}\normpiDE{ g - \pi_\epsilon g} \\
        &\qquad\qquad + 2^{3/2}\rho_2^{t+s/2}\normpiE{\mu-\pi} \norm{f -\pi_\epsilon f}_{L_4'(\pi_\epsilon)}\norm{ g - \pi_\epsilon g}_{L_4'(\pi_\epsilon)}\\
        &\qquad\qquad -(\mu P_\epsilon^t f-\pi_\epsilon f)\left(\mu P_\epsilon^{t+s} g -\pi_\epsilon g\right)
\]
\end{itemize}
\end{corollary}
\begin{proof}
This will use the following shorthand notation. Let
\*[
f_0 &= f-\pi f\\
g_0 &= g-\pi g\\
\Vert h\Vert_\star &= \left(\int (h(x)-\pi h)^2\pi(d x)\right)^{1/2}\\
\Vert h\Vert_\dstar &= \left(\int (h(x)-\pi h)^4\pi(d x)\right)^{1/4}\\
C_\mu &= \Vert \mu -\pi\Vert_2
\]

$\Vert\cdot\Vert_\dstar$ can be interpreted as a centred 4-norm. It is certainly bounded above by $\Vert\cdot\Vert_4$, the norm on $L'_4(\pi)$. For some results regarding the properties of a Markov transition kernel as an operator on $L'_p(\pi)$ for general $p$ given an $L_2$-spectral gap (as is implied by $L_2$-geometric ergodicity) please refer to \citet{rudolf2011explicit}.

\label{pf:fgcor2}
We only show the proof for the original chain. The result for the perturbed chain has essentially the same proof.

By definition we can express the covariance by the triple integral below. We re-express this integral as a sum of two integrals involving the chain run from stationarity. This will allow us to apply \cref{fgcor}.
\*[
&\Cov(f(X_t),g(X_{t+s}))\\
     &\ = \iiint (f(y)-\mu P^t f)(g(z)-\mu P^{t+s} g)\mu(d x)P^t(x,d y)P^s(y,dz)\\
     &\ = \iiint (f(y)-\mu P^t f)(g(z)-\mu P^{t+s} g)\left[\frac{d\mu}{d\pi}(x)-1\right] \pi(d x)P^t(x,d y)P^s(y,dz)\hspace{2em}\\
     &\ \quad +\iiint(f(y)-\mu P^t f)(g(z)-\mu P^{t+s} g)\pi(d x)P^t(x,d y)P^s(y,dz)
\]
We will simplify each of these expressions separately, starting with the second term:
\*[
&\iiint(f(y)-\mu P^t f)(g(z)-\mu P^{t+s} g)\pi(d x)P^t(x,d y)P^s(y,dz)\\
    &\qquad = \iint(f(y)-\mu P^t f)(g(z)-\mu P^{t+s} g)\pi(d y)P^s(y,dz)\\
    &\qquad = \iint f(y)g(z)\pi(d y)P^s(y,dz) \\
        &\qquad \qquad - (\mu P^t f)(\pi g)-(\pi f)(\mu P^{t+s} g)+(\mu P^t f)(\mu P^{t+s} g)\\
    &\qquad = \iint f_0(y)g_0(z) \pi(d y)P^s(y,dz) +(\pi f)(\pi g)\\
    &\qquad \qquad - (\mu P^t f)(\pi g)-(\pi f)(\mu P^{s+t} g)+(\mu P^t f)(\mu P^{t+s} g)\\
    &\qquad = \left\langle f_0, F^s g_0 \right\rangle  +(\mu P^t f-\pi f )(\mu P^{s+t} g - \pi g )
\]

For the first term we find that:
\*[
&\iiint (f(y)-\mu P^t f)(g(z)-\mu P^{t+s} g)\left(\frac{d\mu}{d\pi}(x)-1\right)\pi(d x)P^t(x,d y)P^s(y,dz)\hspace{2em}\\
    &\qquad = \iiint f(y)g(z)\left(\frac{d\mu}{d\pi}(x)-1\right)\pi(d x)P^t(x,d y)P^s(y,dz)\\
        &\qquad\qquad-(\mu P^t f)\iint g(z)\left(\frac{d\mu}{d\pi}(x)-1\right)\pi(d x)P^{t+s}(x,dz)\\
        &\qquad\qquad-(\mu P^{s+t} g)\iint f(y)\left(\frac{d\mu}{d\pi}(x)-1\right)\pi(d x)P^{t}(x,d y)\\
        &\qquad\qquad+(\mu P^t f)(\mu P^{s+t} g)\int\left(\frac{d\mu}{d\pi}(x)-1\right)\pi(d x)\\
    &\qquad = \iiint f_0(y) g_0(z)\left(\frac{d\mu}{d\pi}(x)-1\right)\pi(d x)P^t(x,d y)P^s(y,dz)\\
        &\qquad\qquad-(\mu P^t f-\pi f)\iint g(z)\left(\frac{d\mu}{d\pi}(x)-1\right)\pi(d x)P^{t+s}(x,dz)\\
        &\qquad\qquad-(\mu P^{s+t} g -\pi g)\iint f(y)\left(\frac{d\mu}{d\pi}(x)-1\right)\pi(d x)P^{t}(x,d y)\\
        &\qquad\qquad -(\pi f)(\pi g)\int\left(\frac{d\mu}{d\pi}(x)-1\right)\pi(d x)\\
    &\qquad =\left\langle\frac{d\mu}{d\pi}-1, F^t (f_0\otimes (F^s g_0))\right\rangle\\
        &\qquad \qquad -(\mu P^t f-\pi f)\left\langle \frac{d\mu}{d\pi}-1,F^{t+s} g\right\rangle-(\mu P^{t+s} g -\pi g)\left\langle \frac{d\mu}{d\pi}-1,F^{t} f\right\rangle\\
    &\qquad =\left\langle\frac{d\mu}{d\pi}-1, F^t (f_0\otimes (F^s g_0))\right\rangle-2(\mu P^t f-\pi f)\left(\mu P^{t+s} g -\pi g\right)\\
\]
Where $f_0\otimes F^s g_0$ is defined by
\*[
    [f_0\otimes F^s g_0](y) = f_0(y) \int g_0(z)P^s(y,dz)
\]
Putting these together,
\*[
    &\Cov(f(X_t),g(X_{t+s}))\\
    &\qquad = \left\langle f_0, F^s g_0 \right\rangle  +(\pi f -\mu P^t f)(\pi g - \mu P^{s+t} g)\\
        &\qquad \qquad + \left\langle\frac{d\mu}{d\pi}-1, F^t (f_0\otimes (F^s g_0))\right\rangle-2(\mu P^t f-\pi f)\left(\mu P^{t+s} g -\pi g\right)\\
    &\qquad = \left\langle f_0, F^s g_0 \right\rangle + \left\langle\frac{d\mu}{d\pi}-1, F^t (f_0\otimes (F^s g_0))\right\rangle -(\mu P^t f-\pi f)\left(\mu P^{t+s} g -\pi g\right)\\
    &\qquad \leq (1-\alpha)^s \Vert f\Vert_\star \Vert g\Vert_\star +     (1-\alpha)^{t}\Vert \mu-\pi\Vert_2 \Vert f_0\otimes F^s g_0\Vert_2\\
        &\qquad\qquad -(\mu P^t f-\pi f)\left(\mu P^{t+s} g -\pi g\right)\\
    &\qquad \leq (1-\alpha)^s \Vert f\Vert_\star \Vert g\Vert_\star +     (1-\alpha)^{t}\Vert \mu-\pi\Vert_2 \Vert f_0 \Vert_4\Vert F^s g_0\Vert_4\\
        &\qquad\qquad -(\mu P^t f-\pi f)\left(\mu P^{t+s} g -\pi g\right)\\
    &\qquad \leq (1-\alpha)^s \Vert f\Vert_\star \Vert g\Vert_\star +     (1-\alpha)^{t}\Vert \mu-\pi\Vert_2 \Vert f_0 \Vert_4\Vert g_0\Vert_4 \left\Vert F^s\big\vert_{L'_{4,0}}\right\Vert_4\\
        &\qquad\qquad -(\mu P^t f-\pi f)\left(\mu P^{t+s} g -\pi g\right)\\
    &\qquad \leq (1-\alpha)^s \Vert f\Vert_\star \Vert g\Vert_\star +     2^{3/2}(1-\alpha)^{t+s/2}\Vert \mu-\pi\Vert_2 \Vert f \Vert_\dstar\Vert g\Vert_\dstar\\
        &\qquad\qquad -(\mu P^t f-\pi f)\left(\mu P^{t+s} g -\pi g\right)\\
\]
The $\left\langle f_0, F^s g_0 \right\rangle$ term is bounded using \cref{fgcor} where we have taken the result in its equivalent form using the $\langle\cdot,\cdot\rangle$ notation and the forward operator $F$.
The $\left\langle\frac{d\mu}{d\pi}-1, F^t (f_0\otimes (F^s g_0))\right\rangle$ term is bounded following the methodology of the proof of \citep{rudolf2011explicit}, Lemma 3.39 (in order the inequalities are:
Cauchy-Schwarz, $\Vert F^s g_0\Vert \leq \Vert F^s\Vert\Vert g_0\Vert$ for any norm $\Vert\cdot\Vert$, and Proposition 3.17 of \citep{rudolf2011explicit}).
\end{proof}

The main motivation in establishing the covariance bounds in \cref{fgcor,fgcor2} is that we will need to sum up covariances in order to establish bounds on the variance component of mean-squared error for estimation of $\pi(f)$ via the dependent sample means $\frac{1}{t}\sum_{j=0}^{t-1} f(X_j)$ and $\frac{1}{t}\sum_{j=0}^{t-1} f(X^\epsilon_j)$ for an arbitrary starting measure. To this end we will be interested in the following summation result.

\begin{corollary} \label{fgcorSum}
Under the assumptions of \cref{sec-assump},
\begin{itemize}
\item[(a)] if $X_0\sim\mu$, then for $f,g\in L'_4(\pi)$
\*[
&\frac{1}{t^2}\sum_{m=0}^{t-1}\sum_{n=0}^{t-1} \Cov (f(X_j),f(X_k)) \\
    &\qquad\leq  \frac{2 \normpiD{f - \pi f}^2}{\alpha t}
        + \frac{2^{7/2} \normpi{\mu -\pi}\norm{f -\pi f}_{L_4'(\pi)}^2}{\alpha^2 t^2}
        -\left(\frac{1}{t}\sum_{m=0}^{t-1}\mu P^m f-\pi f\right)^2
\]
\item[(b)] if $\epsilon<\alpha$, and $P_\epsilon$ is $\pi_\epsilon$-reversible, $\rho_2 = (1-(\alpha-\epsilon))$, and $X^\epsilon_0\sim\mu$ , then for $f,g\in L'_4(\pi_\epsilon)$
\*[
&\frac{1}{t^2}\sum_{m=0}^{t-1}\sum_{n=0}^{t-1} \Cov (f(X_j^\epsilon),f(X_k^\epsilon)) \\
    &\qquad\leq  \frac{2 \normpiDE{f - \pi f}^2}{(1-\rho_2) t}
        + \frac{2^{7/2} \normpiE{\mu -\pi}\norm{f -\pi f}_{L_4'(\pi_\epsilon)}^2}{(1-\rho_2)^2 t^2}
        -\left(\frac{1}{t}\sum_{m=0}^{t-1}\mu P_\epsilon^m f-\pi f\right)^2
\]
\end{itemize}
\end{corollary}
\begin{proof}
We only show the proof for the original chain. The results for the perturbed chain have essentially the same proof. The proof is largely an exercise in summation of geometric series and meticulous bookkeeping. The first inequality is due to \cref{fgcor2}. The second inequality makes use of the fact $0<\alpha<1$. To simplify notation, $C_\mu = \normpi{\mu-\pi}$.

	\*[
	    &\frac{1}{t^2}\sum_{m=0}^{t-1}\sum_{n=0}^{t-1} \Cov (f(X_j),f(X_k))\\
	    &\qquad = \frac{\Vert f\Vert_\star^2}{t^2}\sum_{m=0}^{t-1}\sum_{n=0}^{t-1} (1-\alpha)^{\vert m - n\vert}
	        -\frac{1}{t^2}\sum_{m=0}^{t-1}\sum_{n=0}^{t-1}(\mu P^m f-\pi f)\left(\mu P^n f -\pi f\right)\\
	        &\qquad\qquad + \frac{2^{3/2} C_\mu \Vert f \Vert_\dstar^2}{t^2}\sum_{m=0}^{t-1}\sum_{n=0}^{t-1} (1-\alpha)^{(m+n)/2}\\
	    &\qquad = \frac{\Vert f\Vert_\star^2}{t^2}\sum_{m=0}^{t-1} \left(1+2\sum_{s=1}^{t-m-1} (1-\alpha)^{s}\right)
	        -\left(\frac{1}{t}\sum_{m=0}^{t-1}(\mu P^m f-\pi f)\right)^2\\
	        &\qquad\qquad+ \frac{2^{3/2} C_\mu \Vert f \Vert_\dstar^2}{t^2}\sum_{m=0}^{t-1}(1-\alpha)^m\left(1+2\sum_{s=1}^{t-m-1} (1-\alpha)^{s/2}\right)\\
	    &\qquad = \frac{\Vert f\Vert_\star^2}{t^2}\sum_{m=0}^{t-1} \left(1+2\frac{(1-\alpha)-(1-\alpha)^{t-m}}{\alpha}\right)
	        -\left(\frac{1}{t}\sum_{m=0}^{t-1}(\mu P^m f-\pi f)\right)^2\\
	        &\qquad\qquad+ \frac{2^{3/2} C_\mu \Vert f \Vert_\dstar^2}{t^2}\sum_{m=0}^{t-1}(1-\alpha)^m\left(1+2\frac{\sqrt{1-\alpha}-\sqrt{1-\alpha}^{t-m}}{1-\sqrt{1-\alpha}}\right)\\
	    &\qquad = \frac{\Vert f\Vert_\star^2}{t^2}\sum_{m=0}^{t-1} \left(\frac{2-\alpha}{\alpha} -\frac{2}{\alpha}(1-\alpha)^{t-m}\right)
	        -\left(\frac{1}{t}\sum_{m=0}^{t-1}(\mu P^m f-\pi f)\right)^2\\
	        &\qquad\qquad+ \frac{2^{3/2} C_\mu \Vert f \Vert_\dstar^2}{t^2}\sum_{m=0}^{t-1}\left((1-\alpha)^m\frac{1+\sqrt{1-\alpha}}{1-\sqrt{1-\alpha}}-2\frac{\sqrt{1-\alpha}^{t+m}}{1-\sqrt{1-\alpha}}\right)\\
	    &\qquad = \frac{\Vert f\Vert_\star^2}{t^2} \left(\frac{2-\alpha}{\alpha}t -\frac{2}{\alpha}\frac{(1-\alpha) - (1-\alpha)^{t+1}}{\alpha}\right)
	        -\left(\frac{1}{t}\sum_{m=0}^{t-1}(\mu P^m f-\pi f)\right)^2\\
	        &\qquad\qquad+ \frac{2^{3/2} C_\mu \Vert f \Vert_\dstar^2}{t^2}\left(\left[\frac{1+\sqrt{1-\alpha}}{1-\sqrt{1-\alpha}}\right]\left[\frac{1-(1-\alpha)^t}{\alpha}\right]\right.\\
	        &\qquad\qquad\hspace{10em}\left.-\left[\frac{2\sqrt{1-\alpha}^t}{1-\sqrt{1-\alpha}}\right]\left[\frac{1-\sqrt{1-\alpha}^{t}}{1-\sqrt{1-\alpha}}\right]\right)
	\]
	\*[
	    &\qquad = (2-\alpha)\frac{\Vert f\Vert_\star^2}{\alpha t}  - 2(1-\alpha)\frac{1-(1-\alpha)^t}{\alpha^2 t^2}
	        -\left(\frac{1}{t}\sum_{m=0}^{t-1}(\mu P^m f-\pi f)\right)^2\\
	        &\qquad\qquad+ \frac{2^{3/2} C_\mu \Vert f \Vert_\dstar^2}{t^2}\left(\frac{1+\sqrt{1-\alpha}}{\alpha}\right)^2(1-(1-\alpha)^{t/2})^2\\
	    &\qquad \leq \frac{2\Vert f\Vert_\star^2}{\alpha t}
	        + \frac{2^{7/2} C_\mu \Vert f \Vert_\dstar^2}{\alpha^2 t^2}
	        -\left(\frac{1}{t}\sum_{m=0}^{t-1}(\mu P^m f-\pi f)\right)^2
	\]

\end{proof}

\subsubsection{Mean Squared Error Bonds}
\label{ax:mse-proofs}

\begin{theorem}
Under the assumptions of \cref{sec-assump},
if $X_0\sim\mu\in L_2(\pi)$, then
\*[
&\mathbb{E}\left[\left(\pi(f) - \frac{1}{t}\sum_{k=0}^{t-1}
f(X_k)\right)^2\right]
    \leq  \frac{2\Vert f - \pi f\Vert_2^2}{\alpha t}
        + \frac{2^{7/2} \Vert \mu -\pi\Vert_2 \Vert f -\pi f \Vert_4^2}{\alpha^2 t^2}
\]
\end{theorem}

\begin{proof}
The proof proceeds by partitioning the MSE via the bias-variance decomposition then bounding variance term and noting that our bond for the variance contains an expression which exactly cancels the bias term.
We compute that
\*[
& \mathbb{E}\left[\left(\pi(f) - \frac{1}{t}\sum_{k=0}^{t-1}
f(X_k)\right)^2\right]
\\
    &=\mathbb{E}\left[\left(\pi(f) - \frac{1}{t}\sum_{k=0}^{t-1}
    [\mu P^k](f) -  \frac{1}{t}\sum_{k=0}^{t-1} (f(X_k)-[\mu
    P^k](f))\right)^2\right]\\
    &=\left(\pi(f) - \frac{1}{t}\sum_{k=0}^{t-1} [\mu P^k](f)\right)^2
    +\mathbb{E}\left[\left(\frac{1}{t}\sum_{k=0}^{t-1} (f(X_k)-[\mu
    P^k](f))\right)^2\right]\\
    &=\left(\pi(f) - \frac{1}{t}\sum_{k=0}^{t-1} [\mu P^k](f)\right)^2
    +\frac{1}{t^2}\sum_{j=0}^{t-1}\sum_{k=0}^{t-1} \Cov
    (f(X_j),f(X_k))
\]
The variance term is bounded using \cref{fgcorSum}:
\*[
&\frac{1}{t^2}\sum_{j=0}^{t-1}\sum_{k=0}^{t-1} \Cov (f(X_j),f(X_k))\\
    &\qquad\frac{2\Vert f - \pi f\Vert_2^2}{\alpha t}
        + \frac{2^{7/2} \Vert \mu -\pi\Vert_2 \Vert f -\pi f \Vert_4^2}{\alpha^2 t^2}
        -\left(\frac{1}{t}\sum_{m=0}^{t-1}\mu P^m f-\pi f\right)^2
\]
Putting these together yields the desired result.
\end{proof}

\begin{remark}\normalfont
\label{midrangeremark}
We note that, as per \cref{hstarremark},
$\Vert f - \pi f\Vert \leq \Vert f\Vert_2 $. Similarly $\Vert f-\pi f\Vert_4\leq \Vert f\Vert_4$. Also in the case that $f$
is is $\pi$-essentially bounded, $\Vert f\Vert_2 \leq \Vert f\Vert_\infty$ and $\Vert f\Vert_4 \leq \Vert f\Vert_\infty$. These alternative norms may be substituted into the result as necessary in order to make the bounds tractable for a given application.
\end{remark}

\begin{remark}\label{compareremark}\normalfont
Comparing our above geometrically ergodic results to the $L_1$ results
of \citep{johndrow2015approximations} in the uniformly ergodic case,
we see that the $L_2$ and $L_1$ bounds we
establish above differ from the corresponding $L_1$ bound
of \citep{johndrow2015approximations}
only by a factor, which is
constant in time, but varies with the initial distribution (as is to be
expected when moving from uniform ergodicity to geometric ergodicity). For
the Mean-Squared-Error results, the $\Vert\cdot\Vert_\star$-norm in
that paper is based on the midrange-centred infinity norm, which
as per \cref{midrangeremark} is an upper bound on what we have.
\end{remark}

\begin{proof}[Proof of \cref{mse_perturb}]
For the first result, we proceed via bias-variance decomposition, as in the corresponding
result for the exact chain.  However, now the bias under consideration is
itself decomposed as the square of a sum of two components. The squared
sum is expanded simultaneously with the bias-variance expansion.
We compute that
\*[
& \mathbb{E}\left[\left(\pi(f) - \frac{1}{t}\sum_{k=0}^{t-1}
f(X^\epsilon_k)\right)^2\right]
\\
    &=\mathbb{E}\left[\left(\pi(f) - \pi_\epsilon(f) +
    \frac{1}{t}\sum_{k=0}^{t-1} \left[\pi_\epsilon - \mu
    P_\epsilon^k\right](f) -  \frac{1}{t}\sum_{k=0}^{t-1}
    (f(X^\epsilon_k)-[\mu P_\epsilon^k](f))\right)^2\right]\\
    &=\left([\pi-\pi_\epsilon](f)\right)^2
    +2\left([\pi-\pi_\epsilon](f)\right)\left(\pi_\epsilon(f) -
    \frac{1}{t}\sum_{k=0}^{t-1} [\mu P_\epsilon^k](f)\right)\\
    &\hspace{2em} + \left(\pi_\epsilon(f) -
    \frac{1}{t}\sum_{k=0}^{t-1} [\mu P_\epsilon^k](f)\right)^2
    +\frac{1}{t^2}\sum_{j=0}^{t-1}\sum_{k=0}^{t-1} \Cov
    (f(X^\epsilon_j),f(X^\epsilon_k))
\]
We bound the first component of the bias term using versions of \cref{lem:l2bound}
\*[
	\left([\pi-\pi_\epsilon](f)\right)^2
			& = \left([\pi-\pi_\epsilon](f-\pi_\epsilon f)\right)^2 \\
			& \leq \begin{cases}
						\normpi{\pi-\pi_\epsilon}^2 \normpiD{f-\pi_\epsilon f}^2 & \\
						\normpiE{\pi-\pi_\epsilon}^2 \normpiDE{f-\pi_\epsilon f}^2 & \\
					\end{cases}\\
			& \leq \begin{cases}
					\frac{\epsilon^2}{\alpha^2 - \epsilon^2} \normpiD{f-\pi_\epsilon f}^2 & \\
					\frac{\epseps^2}{(1-\rho_2)^2 - \epseps^2}  \normpiDE{f-\pi_\epsilon f}^2 &:\text{given } (*)
				\end{cases}
\]

We bound the variance term using \cref{fgcorSum}:
\*[
& \frac{1}{t^2}\sum_{j=0}^{t-1}\sum_{k=0}^{t-1} \Cov
(f(X^\epsilon_j),f(X^\epsilon_k))
\\
    &\qquad\leq  \frac{2\normpiDE{ f - \pi_\epsilon f}^2}{(1-\rho_2) t}
        + \frac{2^{7/2} \normpiE{ \mu -\pi_\epsilon} \norm{f -\pi_\epsilon f }_{L_4'(\pi_\epsilon)}}{(1-\rho_2)^2 t^2}
        -\left(\frac{1}{t}\sum_{m=0}^{t-1}\mu P_\epsilon^m f-\pi_\epsilon f\right)^2
\]
The negative term in this expression exactly cancels out the third bias term in the expansion.

Finally, we bound the second bias term using
\cref{lem:l2bound,perturberrorthm}:
\*[
	& \hspace{-1em} 2\lcr({[\pi-\pi_\epsilon](f)})\lcr({\pi_\epsilon(f) -
\frac{1}{t}\sum_{k=0}^{t-1} [\mu P_\epsilon^k](f)})\\
    & = 2\lcr({[\pi-\pi_\epsilon](f-\pi_\epsilon f)})\lcr({\lcr[{\pi_\epsilon-\frac{1}{t}\sum_{k=0}^{t-1} \mu P_\epsilon^k }](f-\pi_\epsilon f)})\\
		& \leq 2 \begin{cases}
				\frac{\epseps}{\sqrt{(1-\rho_2)^2 - \epseps^2}} \normpiD{f-\pi_\epsilon f} \frac{1-(1-(\alpha-\epsilon))^t}{t(\alpha-\epsilon)}\normpi{
		    \pi_\epsilon-\mu} \normpiD{f-\pi_\epsilon f} & \\
				\frac{\epseps^2}{(1-\rho_2)^2 - \epseps^2}  \normpiDE{f-\pi_\epsilon f}^2 \frac{1-\rho_2^t}{t(1-\rho_2)}\normpiE{
				\pi_\epsilon-\mu} \normpiDE{f-\pi_\epsilon f}
					&:\text{given } (*)
			\end{cases} \\
			& \leq 2 \begin{cases}
				\frac{\epsilon}{\sqrt{\alpha^2 - \epsilon^2}}\frac{1}{t(\alpha-\epsilon)}\normpi{\pi_\epsilon-\mu} \normpiD{f-\pi_\epsilon f}^2 & \\
				\frac{\epseps}{\sqrt{(1-\rho_2)^2 - \epseps^2}}  \frac{1}{t(1-\rho_2)}\normpiE{
				\pi_\epsilon-\mu}\normpiDE{f-\pi_\epsilon f}^2
						&:\text{given } (*)
				\end{cases}
\]
Putting these together yields the first and third results.

For the second and fourth result we use the fact that for any random variable, $Z$, and for any $a,b\in\RR$ the following holds:
\*[
\mathbb{E}[(Z-a)^2]
    &= 2\mathbb{E}[(Z-b)^2]+2(a-b)^2 - \mathbb{E}[(Z+a-2 b)^2]\\
    &\leq 2\mathbb{E}[(Z-b)^2]+2(a-b)^2
\]
\*[
& \mathbb{E}\lcr[{\lcr({\pi(f) - \frac{1}{t}\sum_{k=0}^{t-1}
f(X^\epsilon_k)})^2}]
\\
    &\qquad\leq 2([\pi-\pi_\epsilon](f))^2 +2 \mathbb{E}\lcr[{\lcr({\pi_\epsilon(f)  - \frac{1}{t}\sum_{k=0}^{t-1}
    f(X^\epsilon_k)})^2}]\\
    &\qquad= 2([\pi-\pi_\epsilon](f-\pi_\epsilon f))^2 \\
    &\qquad\qquad+2 \mathbb{E}\lcr[{\lcr({\pi_\epsilon(f)-\frac{1}{t}\sum_{k=0}^{t-1}
    [\mu P_\epsilon^k](f)  - \frac{1}{t}\sum_{k=0}^{t-1}
    (f(X^\epsilon_k)-[\mu P_\epsilon^k](f))})^2}]\\
    &\qquad= 2([\pi-\pi_\epsilon](f-\pi_\epsilon f))^2
        +2\lcr({\pi_\epsilon(f)-\frac{1}{t}\sum_{k=0}^{t-1}[\mu P_\epsilon^k](f)})^2\\
        &\qquad\qquad+2\mathbb{E}\lcr[{\lcr({\frac{1}{t}\sum_{k=0}^{t-1}
    (f(X^\epsilon_k)-[\mu P_\epsilon^k](f))})^2}] \\
    &\qquad= 2([\pi-\pi_\epsilon](f-\pi_\epsilon f))^2
        +2\lcr({\pi_\epsilon(f)-\frac{1}{t}\sum_{k=0}^{t-1}[\mu P_\epsilon^k](f)})^2\\
        &\qquad\qquad+\frac{2}{t^2}\sum_{j=0}^{t-1}\sum_{k=0}^{t-1} \Cov
    (f(X^\epsilon_j),f(X^\epsilon_k))
\]
Applying \cref{fgcor2} to bound the sum of covariances, we find that we are able to exactly cancel the second term in the final expression above. Using the same bound as before for the first expression, we get the final result.
\end{proof}

%% file: section-files/geomPerturb_apdx-mcmc-proofs.tex
Let
\*[
\gamma(x)
    &= \mathbb{E}_{y\sim q(y|x)} r(y|x)
    = \int r(y|x)q(y|x) d y \\
[\nu\Gamma]( d y )
    &= \nu(y)\gamma(y) d y \\
[\nu Z]( d y )
    &= \left[\int r(y|x) q(y|x)\nu(x) d x \right] d y
\]

\begin{lemma}
$P-\hat P = Z -\Gamma$
\end{lemma}
\begin{proof}
We first give expressions for the elements of measure for transitions
of the original chain. The first formula is the element of measure for
transition from an arbitrary, fixed initial point. It is defined for
us by the mechanics of the Metropolis--Hastings algorithm. The second
expression is the element of measure for transition from a sample from
an initial distribution, $\nu$. It is derived from the first expression
by integrating over the sample from $\nu$.
\*[
P(x,  d x' )
    &= \delta_x(  d x' ) \left[1-\int({{a}}(y|x)q(y|x) d y  \right]
    + {{a}}(x'|x) q(x'|x)  d x' \\
\left[\nu P\right](  d x' )
    &= \int\left[\delta_x(  d x' ) \left[1-\int{{a}}(y|x)q(y|x) d y
    \right] + {{a}}(x'|x) q(x'|x)  d x' \right]\nu(x) d x \\
    &= \left[\left[1-\int{{a}}(y|x')q(y|x') d y  \right]\nu(x')
    + \int{{a}}(x'|x) q(x'|x)\nu(x) d x \right]  d x' \\
\]
The second form of the second expression is an application of Fubini's
theorem. The exchange of the order of integration for the second term in
the expression is immediate. For the first term, for arbitrary non-negative functions $f$,
\*[
    \int_s\int_t f(s,t) \delta_t( d s ) d t  =\int_t\int_s  f(s,t)
\delta_t( d s ) d t  = \int_t  f(t,t)  d t =\int_s  f(s,s) d s
\]
Where the first equality is Fubini's theorem, the second comes from
integrating with respect to $s$, and the third comes from a change of
dummy variable.

Similarly, the elements of measure for transitions from the approximating
kernel are expressed below. The first expression, as above, is the element
of measure for transition from an arbitrary, fixed initial point. It
is defined for us by the mechanics of the noisy Metropolis--Hastings
algorithm. The second expression is again derived by integrating the
first against an initial measure, $\nu$.
\*[
\hat P(x,  d x' )
    &= \delta_x(  d x' ) \left[1-\iint{\hat{a}}(y|x,z)q(y|x)f_y(z)
     d z  d y  \right] \\
    &\qquad\qquad + \int {\hat{a}}(x'|x,z)
    q(x'|x)f_{x'}(z) d z  d x' \\
\left[\nu \hat P\right](  d x' )
    &= \int\bigg(\delta_x(  d x' )
    \left[1-\iint{\hat{a}}(y|x,z)q(y|x)f_y(z)
     d z  d y  \right] \\
&\qquad\qquad\qquad\qquad + \int {\hat{a}}(x'|x,z)
    q(x'|x)f_{x'}(z) d z  d x' \bigg)\nu(x) d x \\
    &= \left[1-\iint{\hat{a}}(y|x',z)q(y|x')f_y(z)  d z
     d y  \right]\nu(x')  d x'  \\
&\qquad\qquad\qquad\qquad + \left[\iint {\hat{a}}(x'|x,z)
    q(x'|x)f_{x'}(z)\nu(x)  d z d x \right]  d x'
\]
The same applications of Fubini's theorem occur as above.

We may now leverage our notation defined above to simplify the difference
of these elements of measure.
\*[
&\left[\nu (P-\hat P)\right] (  d x' )\\
    &\qquad=
    \left[\iint\bigg({\hat{a}}(y|x',z)-{{a}}(y|x')\bigg)q(y|x')f_y(z)
     d z  d y \right]\nu(x')  d x'  \\
    &\qquad\hspace{2em} + \left[\iint
    \bigg({{a}}(x'|x)-{\hat{a}}(x'|x,z)\bigg)
    q(x'|x)f_{x'}(z)\nu(x)  d z d x \right]  d x' \\
    &\qquad=\left[\int r(x'|x) q(x'|x)\nu(x)  d x \right]  d x'  -
    \left[\int r(y|x') q(y|x')  d y \right]\nu(x')  d x' \\
    &\qquad=[\nu(Z-\Gamma)](  d x' )
\]
From this one may conclude that $\left(P- \hat P  =  Z - \Gamma\right)$
as operators.
\end{proof}

\begin{proof}[Proof of \cref{noisyoperatorthm}]
It is obvious that if $\abs{r(y|x)}\leq R$ uniformly in $(x,y)\in \Xx^2$
then
\[\left(\normpi{\Gamma} \leq R\right) \ ,\] and
\[\left(\normpi{Z}\leq R \normpi{Q}\right) \ .\]
By applying the previous
lemma, given the assumptions stated,
\[\normpi{P-\hat P} \leq R
(1+\normpi{Q}) \ .\]
\end{proof}

%% file: section-files/geomPerturb_apdx-example-weakl2rate.tex
Let $\Xx = \NatsO$, and let $a$ be a probability mass function on $\Xx$. Define transition probabilities by
\[
  p_{ij}
    & = \begin{cases}
    a_j
      &: i = 0 \\
    1
      &: i>0,\ j=i-1 \\
    0
      &: \text{otherwise}
    \end{cases}
\]
Let $b_j = \sum_{i=j}^\infty a_j$. It is easy to verify that if $\sum_{j=1}^\infty b_j < \infty$ then $\pi_j = \frac{b_j}{\sum_{j=1}^\infty b_j}$ is the unique stationary probability mass function for $P = [p_{ij}]_{ij\in\Xx^2}$.

In the special case where $a_j = 2^{-j-1}$, we have $\pi = a$. We continue this example working exclusively with this choice of $a$. Now,
\[
  \delta_j P^n
    & = \begin{cases}
      \pi
        &: n\geq j+1 \\
      \delta_{n-j}
        &: n\leq j
  \end{cases}
\]
Thus, for any initial probability mass function, $\mu$,
\[
  [\mu P^n]_j = \sum_{i=0}^{n-1} \mu_i \pi_j + \mu_{j+n}
\]
If $\rnderiv{\mu}{\pi}(j) = \frac{\mu_j}{\pi_j} \leq \normpiInf{\mu} <\infty$ for all $j\in\Xx$ then
\[
  \normpi{\mu P^n - \pi}^2
    & = \sum_{j=0}^\infty \pi_j \lcr({\sum_{i=0}^{n-1} \mu_i + \frac{\mu_{j+n}}{\pi_j} -1 })^2 \\
    & = \sum_{j=0}^\infty \pi_j \lcr({-\sum_{i=n}^{\infty} \mu_i + \frac{\mu_{j+n}}{\pi_{j+n}} \frac{\pi_{j+n}}{\pi_j}})^2\\
    & =  \sum_{j=0}^\infty \pi_j \lcr({-\sum_{i=n}^{\infty} \frac{\mu_i}{\pi_i} \pi_i + \frac{\mu_{j+n}}{\pi_{j+n}} \frac{\pi_{j+n}}{\pi_j}})^2\\
    &\leq  \sum_{j=0}^\infty \pi_j \lcr({\sum_{i=n}^{\infty} \frac{\mu_i}{\pi_i} \pi_i + \frac{\mu_{j+n}}{\pi_{j+n}} \frac{\pi_{j+n}}{\pi_j}})^2\\
    &\leq \sum_{j=0}^\infty 2^{-j-1} \lcr({\sum_{i=n}^{\infty} \normpiInf{\mu} 2^{-i-1} + \normpiInf{\mu} 2^{-n}})^2\\
    & = \normpiInf{\mu}^2 \sum_{j=0}^\infty 2^{-j-1} (2^{-n+1})^2\\
    & = 4 \normpiInf{\mu}^2 (2^{-n})^2
\]
Hence $P$ is $(L_\infty(\pi),\normpi{\cdot})$-GE with optimal rate no larger than $1/2$.

For any $\alpha <\sqrt{0.5}$, let $\nu_j = (1-\alpha) (\alpha)^j$.  Then $\nu\in L_2(\pi)$, since
\[
  \normpi{\nu}^2
    & = \sum_{i=0}^\infty 0.5^{i+1} \lcr({\frac{(1-\alpha)(\alpha)^i }{0.5^{i+1}}})^2 \\
    & = 2(1-\alpha)^2 \sum_{i=0}^\infty (2\alpha^2)^i
        = \frac{2(1-\alpha)^2}{1 - 2\alpha^2}
\]
Moreover,
\[
  \normpi{\nu P^n - \pi}^2
    & =  \sum_{j=0}^\infty \pi_j \lcr({\sum_{i=0}^{n-1} \nu_i + \frac{\nu_{j+n}}{\pi_j} -1 })^2 \\
    & =  \sum_{j=0}^\infty 0.5^{j+1} \lcr({-\sum_{i=n}^{\infty} (1-\alpha)\alpha^i + (1-\alpha)\alpha^{j+n} (0.5)^{-j-1} })^2 \\
    & =  \alpha^{2n}\sum_{j=0}^\infty 0.5^{j+1} \lcr({-1 + 2(1-\alpha)(2\alpha )^j})^2 \\
    & =  \frac{\alpha^{2n}}{2} \sum_{j=0}^\infty ( 0.5^{j} -4(1-\alpha)\alpha^j +4(1-\alpha)^2 (2\alpha^2)^j)  \\
    & =  \frac{\alpha^{2n}}{2} \lcr({2 - \frac{4(1-\alpha)}{1-\alpha} + \frac{4(1-\alpha)^2}{1-2\alpha^2 }}) \\
    & = \frac{(2\alpha-1)^2}{1-2\alpha^2} \alpha^{2n}
\]
Thus the convergence rate starting from this initial measure is $\alpha$.

Since this is true for any $\alpha <1/\sqrt{2}$, this shows that the $L_2(\pi)$-GE optimal rate is no smaller than $\sqrt{0.5}$.
Hence the $(L_\infty(\pi),\normpi{\cdot})$-GE and $L_2(\pi)$-GE optimal rates are different.

%% file: section-files/geomPerturb_apdx-roberts-tweedie-proof.tex
\begin{proof}[Proof of \cref{lem:opt-rate-char}]
	Let
	\[
		\rho^\star
			& = \inf\lcr\{{\rho>0: \exists C:V\to\PosReals \st \forall n\in\Nats, \nu\in V\cap \Mm_{+,1} \quad \Norm{\nu P^n - \pi} \leq C(\nu) \rho^n}\} \ , \\
		\hat \rho
			& = \sup_{\mu\in V\cap \Mm_{+,1}} \limsup_{n\to\infty} \Norm{\mu P^n -\pi}^{1/n}
\]

	($\hat \rho \leq \rho^\star$):
	Let $\epsilon>0$
	\[
	\hat \rho
		& = \sup_{\mu\in V\cap \Mm_{+,1}} \limsup_{n\to\infty} \Norm{\mu P^n -\pi}^{1/n} \\
		& \leq \sup_{\mu\in V\cap \Mm_{+,1}} \limsup_{n\to\infty} \Norm{\mu P^n -\pi}^{1/n} \\
		& \leq \sup_{\mu\in V\cap \Mm_{+,1}} \limsup_{n\to\infty}  (C_\epsilon(\mu) (\rho^\star+\epsilon)^n)^{1/n}\\
		& =\rho^\star +\epsilon \ .
	\]
	Since $\epsilon$ is arbitrary, $\hat \rho \leq \rho^\star$.

	($\hat \rho \geq \rho^\star$):
	For all $\nu\in V\cap \Mm_{+,1}$, $\limsup_{n\to\infty} \Norm{\mu P^n -\pi}^{1/n}\leq \hat\rho$.
	Let $\epsilon>0$.
	Then for all $\mu\in V\cap \Mm_{+,1}$, $\Norm{\mu P^n -\pi}^{1/n}> \hat\rho+\epsilon$ for at most finitely many $n\in \Nats$. Let $C_\epsilon(\mu) = \max_{n\in\Nats}\rbra{1\vee \frac{\Norm{\mu P^n -\pi}}{(\rho+\epsilon)^n}}$.
	Then $C_\epsilon(\mu) <\infty$ since the maximum is over finitely many distinct elements.
	Therefore $\Norm{\mu P^n - \pi } \leq C_\epsilon(\mu) (\hat\rho+\epsilon)^n$ for all $n\in\Nats$. This implies that $\hat\rho+\epsilon \geq \rho^\star$. Since $\epsilon$ is arbitrary, $\hat\rho\geq \rho^\star$.
\end{proof}

\begin{proof}[Proof of \cref{lem:rt-comment}]
  [\textit{(iii)} $\iff$ \textit{(iv)}] is proven in \citep[Theorem 2.1]{roberts1997geometric}.
  [\textit{(iii)} $\implies$ \textit{(ii)}] follows from the inclusion $L_\infty(\pi)\subset L_2(\pi)$.
  [\textit{(ii)} $\implies$ \textit{(i)}] follows from Cauchy-Schwarz.

  [\textit{(ii)} $\implies$ \textit{(iii)}]:

  Without loss of generality, we may assume that $\rho$ is the optimal rate of $(L_\infty(\pi_\epsilon), \norm{\cdot}_{L_2(\pi)})$-geometric ergodicity;
  \[
    \rho = \sup_{\nu\in L_{\infty,0}(\pi)} \limsup_{t\to\infty}  \normpi{\nu P^t}^{1/t} \ .
  \]

  From the proof of \citet[Theorem~1]{roberts2001geometric}, $P$ is $\pi$-almost-everywhere geometrically ergodic with some unknown optimal rate.
  From \citep[Theorem 2.1]{roberts1997geometric}, $P$ is $L_2(\pi)$-geometrically ergodic with some unknown optimal rate, $\rho_2$, which is equivalent to the spectral radius of $P\restrict{L_{2,0}(\pi)}$; $\rho_2 = r(P\restrict{L_{2,0}(\pi)})$.

  It remains to be shown that $\rho_2 \leq \rho$.
  We will use the spectral measure decomposition of $P$, as in \citep{roberts1997geometric}.
  Suppose, for a contradiction, that $\rho_2 >\rho$.
  Let $\ol \rho = \frac{\rho+\rho_2}{2}$.
  Let $\Ee$ be the spectral measure of $P$, so that $\mu P^t = \int_{-1}^1 \lambda^t \mu \Ee(d\lambda)$.
  If $\rho_2 > \rho$ then either $\Ee([-\rho_2,-\ol \rho))\neq\zero $ or $\Ee((\ol \rho,\rho_2]))\neq \zero$.
  Assume (replacing $P$ by $P^2$, $\rho$ by $\rho^2$, and $\rho_2$ by $\rho_2^2$ if necessary) that $\Ee((\ol \rho,\rho_2])\neq \zero$ and  $\Ee((-1,0))= \zero$.
  Then there is some non-zero signed measure, $\nu$, in the range of $\Ee((\ol \rho,\rho_2])$.
  Since the spectral projections are orthogonal and $\set{1}\cap (\ol \rho,\rho_2] = \emptyset$, then $\nu\perp \pi$, and hence $\nu(\Xx) = 0$.
  Since $L_{\infty,0}(\pi)$ is dense in $L_{2,0}(\pi)$, there is a $\mu\in L_{\infty,0}(\pi)$ with $\normpi{\mu-\nu}<\normpi{\nu}/2$.
  Then, from the polarization identity, $\inner{\nu}{\mu}_{L_2(\pi)} \geq \frac{3}{8} \normpi{\nu}^2 >0$,  and $\mu\neq 0$.

  Let $R = \range(\int_{(\ol \rho,\rho_2]} \Ee(d \lambda))$. Then $\text{span}(\nu) \subset R$, so
  \[
  \normpi{\proj_{R}{\mu}}
    \geq \normpi{\proj_{\nu}{\mu}}
    \geq \frac{3}{8} \normpi{\nu} \\
  \]
  Then
  \[
    \normpi{\mu P^k}^2
      & = \inner{\mu P^k}{\mu P^k}_{L_2(\pi)} \\
      & = \inner{\mu}{\mu P^{2k}}_{L_2(\pi)} \\
      & = \inner{\mu}{\mu\int_{(0,\rho_2]} \lambda^{2k} \Ee(d \lambda)}_{L_2(\pi)} \\
      & \geq \inner{\mu}{\mu\int_{(\ol \rho,\rho_2]} \lambda^{2k} \Ee(d \lambda)}_{L_2(\pi)} \\
      & \geq \inner{\mu}{\mu\int_{(\ol \rho,\rho_2]} \ol \rho^{2k} \Ee(d \lambda)}_{L_2(\pi)} \\
      & = \ol\rho^{2k} \normpi{\proj_{R}{\mu}}^2 \\
      & \geq \ol\rho^{2k} \frac{9}{64} \normpi{\nu}^2 \ .
  \]
  Hence $\rho \geq \ol \rho$. This contradicts $\rho_2 > \rho$.

[\textit{(i)} $\implies$ \textit{(ii)}]:

Let the optimal rates of $(L_\infty(\pi_\epsilon), \norm{\cdot}_{L_1(\pi)})$-GE and $(L_\infty(\pi_\epsilon), \norm{\cdot}_{L_2(\pi)})$-GE be (respectively)
\[
  \rho
    & = \sup_{\mu\in L_{\infty,0}(\pi)}\limsup_{n\to\infty} \normpiO{\mu P^n}^{1/n} \ , &
  \rho_2
    & = \sup_{\mu\in L_{\infty,0}(\pi)}\limsup_{n\to\infty} \normpi{\mu P^n}^{1/n} \ .
\]
We want to show that $\rho_2 \leq \rho$.

Let $\epsilon>0$ be arbitrary. Let $\nu_\epsilon\in L_{\infty,0}(\pi)$ with
\[
  \limsup_{n\to\infty} \normpi{\nu_\epsilon P^n}^{1/n}
    & \geq \rho_2-\epsilon \ .
\]
Then, for some $c(\nu_\epsilon)>0$, for infinitely many $n\in\NN$
\[
  \normpi{\nu_\epsilon P^n}
    & \geq (\rho_2-2\epsilon)^n \ .
\]

Using the fact that $\normpiO{\mu} = \sup_{\substack{f\in L'_\infty(\pi) \\ \norm{f}_{L'_\infty(\pi)}}} \mu f$, and
using the self-adjointness of $P$ in $L_2(\pi)$ (since $P$ is reversible),
and using the fact that (a version of) $\rnderiv{\nu_\epsilon}{\pi}$ is some bounded function with $\norm{\rnderiv{\nu_\epsilon}{\pi}}_{L'_\infty(\pi)} = \norm{\nu_\epsilon}_{L_\infty(\pi)}$,
then for infinitely many $n\in\Nats$,
\[
\normpiO{\nu_\epsilon P^{2n}}
& = \sup_{\norm{f}_\infty\leq 1} \nu_\epsilon P^{2n} f \\
& \geq \frac{1}{\normpiInf{\nu_\epsilon}} \nu_\epsilon P^{2n} \rnderiv{\nu_\epsilon}{\pi} \\
& = \frac{1}{\normpiInf{\nu_\epsilon}} \inner{\nu_\epsilon P^{2n}}{\nu_\epsilon} \\
& = \frac{1}{\normpiInf{\nu_\epsilon}} \inner{\nu_\epsilon P^{n}}{\nu_\epsilon P^n} \\
& = \frac{1}{\normpiInf{\nu_\epsilon}} \normpi{\nu_\epsilon P^n}^2 \\
& \geq \frac{1}{\normpiInf{\nu_\epsilon}} (\rho_2-2\epsilon)^{2n}
\]
Thus $\rho_2-2\epsilon \leq \rho$. Since $\epsilon$ was arbitrary, we find that $\rho_2\leq \rho$.

\end{proof}